\documentclass[12pt,leqno]{amsart}
\usepackage{amsmath, amssymb, amsthm, bm}
\usepackage{rsfso}
\usepackage[svgnames]{xcolor}
\usepackage{tikz}
\usetikzlibrary{arrows, cd, intersections, decorations.pathmorphing}
\tikzset{>=latex}
\usepackage{mathtools}
\usepackage[alwaysadjust]{paralist}
\usepackage[alphabetic]{amsrefs}
\usepackage[T1]{fontenc}
\usepackage{mathptmx}
\usepackage{microtype}
\usepackage[format=hang, font=normalsize, justification=raggedright,
  labelfont=bf]{caption}
\usepackage[colorlinks = true,
            linkcolor  = DarkBlue,
            urlcolor   = DarkGreen,
            citecolor  = DarkBlue]{hyperref}
\hypersetup{
  pdftitle={Smooth Quot Schemes},
  pdfauthor={Roy Skjelnes, Gregory G. Smith, and Mike Stillman}
  }
\usepackage[centering, includeheadfoot, hmargin=1.0in, tmargin=1.0in, 
  bmargin=1in, headheight=6pt]{geometry}
\newtheorem{lemma}{Lemma}[section]
\newtheorem{theorem}[lemma]{Theorem}
\newtheorem{corollary}[lemma]{Corollary}
\newtheorem{proposition}[lemma]{Proposition}
\theoremstyle{definition}

\newtheorem{remark}[lemma]{Remark}
\newtheorem{examplex}[lemma]{Example}
\newenvironment{example}
  {\pushQED{\qed}\examplex}
  {\popQED\endexamplex}

\renewcommand{\AA}{\ensuremath{\mathbb{A}}} 
\newcommand{\PP}{\ensuremath{\mathbb{P}}} 
\newcommand{\QQ}{\ensuremath{\mathbb{Q}}}
\newcommand{\ZZ}{\ensuremath{\mathbb{Z}}}
\newcommand{\kk}{\Bbbk}
\DeclareMathOperator{\coker}{Coker}
\DeclareMathOperator{\image}{Im}
\DeclareMathOperator{\Ext}{Ext}
\DeclareMathOperator{\Gr}{Gr}
\DeclareMathOperator{\GL}{GL}
\DeclareMathOperator{\Hilb}{Hilb}
\DeclareMathOperator{\Hom}{Hom}
\DeclareMathOperator{\Quot}{Quot}
\DeclareMathOperator{\Spec}{Spec}

\DeclareMathOperator{\sHom}{\mathcal{H} \kern-1.0pt \textit{om}}

\begin{document}

\vspace*{-2.5em}

\title[Smooth Quot Schemes]{Classifying Smooth Quot Schemes}
\author[R.~Skjelnes]{Roy Skjelnes}
\address{Roy Skjelnes: 
    Department of Mathematics, 
    Royal Institute of Technology (KTH), 
    Stockholm, 100~44, Sweden;
    {\normalfont \texttt{skjelnes@math.kth.se}}}

\author[G.G.~Smith]{Gregory G.{} Smith}
\address{Gregory G.{} Smith: 
    Department of Mathematics and Statistics, 
    Queen's University, 
    Kingston, Ontario, K7L 3N6, Canada; 
    {\normalfont \texttt{ggsmith@mast.queensu.ca}}}

\author[M.~Stillman]{Michael Stillman}
\address{Michael Stillman: 
    Department of Mathematics, 
    Cornell University, 
    Ithaca, New York, 14853, United States of America;
    {\normalfont \texttt{mes15@cornell.edu}}}


\begin{abstract}
  The Quot scheme
  $\Quot^{\kern0.5pt q} \kern-0.5pt (\mathcal{O}_{\PP^n}^{\! r})$ parametrizes
  the quotients of the trivial vector bundle of rank $r$ on $n$-dimensional
  projective space that have Hilbert polynomial $q$ and are flat over a base
  scheme.  We identify numerical conditions on the polynomial $q$ that
  completely determine when this Quot scheme is smooth and irreducible.  Our
  approach also uncovers further geometric features of the projective scheme
  $\Quot^{\kern0.5pt q} \kern-0.5pt (\mathcal{O}_{\PP^n}^{\! r})$ including
  the smoothness of the lexicographic point.
\end{abstract}

\maketitle

\vspace*{-1.75em}
\section{Overview}
\label{s:over}

\noindent
Quot schemes, introduced in \cite{Gro61}*{\S3}, are fundamental parameter
spaces in algebraic geometry.  Despite their central role in the construction
of moduli spaces~\cite{HL10}*{\S2.2} and in deformation
theory~\cite{Ser06}*{\S4.4}, most of their geometric properties are still
poorly understood. To circumvent the challenges posed by their singularities,
\cite{CK01}*{\S4} develops the concept of derived Quot schemes.  However, much
of the existing literature on Quot schemes focuses exclusively on the special
case of Hilbert schemes.  For instance, the geometry of
$\Quot^{\kern0.5pt q} \kern-0.5pt (\mathcal{O}_{\PP^n}^{\! r})$, which
parametrizes quotients of the trivial vector bundle
$\smash{\mathcal{O}_{\PP^n}^{\! r}} \coloneqq \smash{\bigoplus_{\! j=1}^{r}
  \mathcal{O}_{\PP^n}^{}}$ that are flat over a base scheme and have Hilbert
polynomial $q$, is much more mysterious when the rank $r$ is greater than
$1$. While \cite{Par96}*{Theorem~36} establishes the path-connectedness of
$\Quot^{\kern0.5pt q} \kern-0.5pt (\mathcal{O}_{\PP^n}^{\! r})$ by extending
the arguments for the Hilbert scheme
$\Hilb^{\kern0.5pt q} \kern-0.5pt (\PP^n) \coloneqq \Quot^{\kern0.5pt q}
\kern-0.5pt (\mathcal{O}_{\PP^{n}}^{\! 1})$, many aspects of Quot schemes and
Hilbert schemes diverge in subtle and surprising ways. This article clarifies
some of these inherent differences.

Our main theorem classifies the smooth
$\Quot^{\kern0.5pt q} \kern-0.5pt (\mathcal{O}_{\PP^n}^{\! r})$. It is
formulated in terms of a family of univariate rational polynomials and has
three distinct cases depending on the rank $r$.  An integer partition
$\lambda$ is an $e$-tuple
$\lambda \coloneqq (\lambda_1, \lambda_2, \dotsc, \lambda_e)$ of integers
satisfying
$\lambda_1 \geqslant \lambda_2 \geqslant \dotsb \geqslant \lambda_{e}
\geqslant 1$ and the associated polynomial is
\[
    p_{\lambda\!}(t) \coloneqq 
    \sum_{i=1}^{e} \binom{t+\lambda_i - i}{\lambda_i -1}
    = \binom{t+\lambda_1 - 1}{\lambda_1 -1} 
    + \binom{t+\lambda_2 - 2}{\lambda_2 -1} + \dotsb 
    + \binom{t+\lambda_e - e}{\lambda_e -1} 
    \in \QQ[t] \, . 
\]

\begin{theorem} 
    \label{t:main}
    Let $n$ and $r$ be positive integers.  For any polynomial $q \in \QQ[t]$,
    the Quot scheme
    $\Quot^{\kern0.5pt q} \kern-0.5pt (\mathcal{O}_{\PP^{n}}^{\! r})$ is
    smooth and irreducible if and only if
    $q(t) = s \, \smash{\binom{t+n}{n}} + p_{\lambda\!}(t)$ for some
    nonnegative integer $s$ and some integer partition
    $\lambda \coloneqq (\lambda_1, \lambda_2, \dotsc, \lambda_e)$ such that
    $r \geqslant s$, $n \geqslant \lambda_1$, and either
    \begin{compactitem}[\phantom{W}$\bullet$]
    \item $r \geqslant s+2$ and one of the following disjoint conditions holds:
    \begin{compactenum}[\phantom{W} \upshape 1.]
        \item  $\lambda_{e} \geqslant 2$,
        \item
          $\lambda = (n^{e-1}, 1) = (\smash{\overbrace{n,n, \dotsc,
              n}^{\text{\upshape $e \! - \! 1$ times}}},1)$ where $n \neq 2$
          and $e \geqslant 2$,
        \item  $\lambda = (2^{e-1}, 1)$ where $n = 2$ and $e \geqslant 4$,
        \item $\lambda = (1)$ or $\lambda = ()$;
    \end{compactenum}
  \item $r = s+1$ and the Hilbert scheme
    $\Hilb^{\kern0.5pt p_{\lambda}} \kern-0.5pt (\PP^n)$ is smooth; or
    \item $r = s$ and $\lambda = ()$.
    \end{compactitem}
\end{theorem}

\noindent
Dovetailing with the special case $r = 1$, the ($r=s+1$)-case can be rephrased
as combinatorial conditions on $\lambda$.  Indeed, \cite{SS23}*{Theorem~A}
already enumerates all integer partitions $\lambda$ for which the Hilbert
scheme $\Hilb^{\kern0.5pt p_{\lambda}} \kern-0.5pt (\PP^n)$ is smooth; see
Remark~\ref{r:Hil} for the details.

This expression for the Hilbert polynomial as a sum of binomial coefficients
already appears in \cite{Got78}*{\S2}; also see
\cite{BH93}*{Theorem~4.3.2}. Proposition~\ref{p:nonempty} relates the pair
$(s,\lambda)$ to the geometry of Quot schemes.

\subsection*{Contrasting Quot schemes and Hilbert schemes}
The differences between smooth Quot schemes and smooth Hilbert schemes are
more pronounced than their similarities.  The classification of smooth Quot
schemes falls into three cases, depending solely on the rank $r$ of the
ambient vector bundle.  The ($r=s$)-case is trivial: The unique point on the
Quot scheme corresponds to the ambient vector bundle
$\mathcal{O}_{\PP^{n}}^{\!r}$ itself. In the ($r=s+1$)-case, the scheme
$\Quot^{\kern0.5pt q} \kern-0.5pt (\mathcal{O}_{\PP^n}^{\! r})$ is locally a
product of a projective space and the Hilbert scheme
$\Hilb^{\kern0.5pt p_{\lambda}} \kern-0.5pt (\PP^n)$; see
Proposition~\ref{p:s+1} and Remark~\ref{r:bun}. Thus, studying these Quot
schemes reduces to understanding Hilbert schemes.

The third case is different.  When $r \geqslant s+2$, smooth Quot schemes are
characterized by four distinct conditions in contrast to seven for Hilbert
schemes; see Remark~\ref{r:Hil}.  As an exemplar of this dichotomy, the
($n=2$)-case for Hilbert schemes warrants a separate condition whereas all
conditions in the ($r \geqslant s+2$)-case are about integer
partitions. Additionally, it is much easier to identify the Quot schemes with
a unique Borel-fixed point than it is for Hilbert schemes; compare
Corollary~\ref{c:!b} and \cite{Sta20}*{Theorem~1.1}.  At first glance, the gap
separating the second and third conditions in the ($r\geqslant s+2$)-case
might resemble the idiopathic singular Hilbert schemes associated to two skew
lines and twisted cubic curves; see \cite{SS23}*{Example~4.4}.  However, these
gaps are not the same.  Remark~\ref{r:gap} analyzes the novel Quot schemes
responsible for this gap and distinguishes them from the relevant Hilbert
schemes.

In summary, smooth Quot and Hilbert schemes coincide in a few special
situations, but the classification when $r \geqslant s+2$ is, surprisingly,
less complicated than the classification when $r = s+1$.

\subsection*{Geometric properties of Quot schemes}
The characterization of smooth Quot schemes hinges on their intrinsic
combinatorial features.  As a benefit, the techniques from combinatorial
algebraic geometry yield new geometric insights into
$\Quot^{\kern0.5pt q} \kern-0.5pt (\mathcal{O}_{\PP^{n}}^{\! r})$. Three
discoveries merit particular attention.

First, Theorem~\ref{t:smooth} establishes that the lexicographic point on
$\Quot^{\kern0.5pt q} \kern-0.5pt (\mathcal{O}_{\PP^{n}}^{\! r})$ is always
smooth, generalizing \cite{RS97}*{Theorem~1.4}. Unlike the proof for Hilbert
schemes, we construct a full-dimensional family on the lexicographic component
without providing an explicit presentation for a corresponding module over the
homogeneous coordinate ring of $\PP^{n}$. When $\lambda = (n^{e-1},1)$,
Proposition~\ref{p:fam} illustrates the power of this approach by highlighting
the unavoidable intricacies in exhibiting a presentation for the family.

As second geometric development, we identify three new countable collections
of nonsingular Quot schemes by relating them to Grassmannians and Hilbert
schemes. Unlike Hilbert schemes, most smooth Quot schemes to the best of our
knowledge do not already appear in the literature. Collectively,
Proposition~\ref{p:gr}, Theorem~\ref{t:prod}, and Proposition~\ref{p:3rd}
demonstrate that
\[
    \Quot^{\kern0.5pt q} \kern-0.5pt (\mathcal{O}_{\PP^{n}}^{\! r})
    \cong \begin{cases}
        \Gr(s,r) 
        & \text{if $\lambda = ()$,} \\[-1pt]
        \PP^{r-1} \times \Hilb^{\kern0.5pt p_{\lambda}}(\PP^n) 
        & \text{if $s = 0$ and $\lambda_{e} \geqslant 2$,} \\[-1pt]
        \PP^{r \binom{n+e}{e} -1}
         & \text{if $r = s+1$ and $\lambda = (n^{e})$,}
    \end{cases}
\]
where $\Gr(s,r)$ denotes the Grassmannian of rank $s$ quotients; see
Section~\ref{s:Gr}. Example~\ref{e:rig} describes the flat deformations of the
tangent bundle on $\PP^n$. These isomorphisms underscore the geometric
significance of the pair $(s, \lambda)$.

Our third geometric contribution provides a thorough analysis of the Quot
schemes determined by the integer partitions $(2,1)$ and $(2,2,1)$ when
$r = 2$ and $s = 0$; see Example~\ref{e:21} and Example~\ref{e:221}.  These
sporadic singular Quot schemes, with their superficial similarities to the
Hilbert schemes containing two skew lines or twisted cubic curves,
unexpectedly deviate from the apparent pattern.

\subsection*{Broader setting and potential generalizations}
Our deliberate focus on the trivial vector bundle $\mathcal{O}_{\PP^n}^{\!r}$
is tantamount to assuming that the quotient sheaves are generated in a single
degree.  This assumption sidesteps some notational and technical issues.  When
the ambient vector bundle is the direct sum
$\bigoplus_{i=1}^{r} \mathcal{O}_{\PP^n}(-d_i)$, for some integers
$d_1 \leqslant d_2 \leqslant \dotsb \leqslant d_r$, the same methods should
apply. For instance, \cite{Par96}*{Theorem~36} establishes the
path-connectedness of
$\Quot^{\kern0.5pt q} \kern-0.5pt \bigl( \bigoplus_{i=1}^{r}
\mathcal{O}_{\PP^n}(-d_i) \kern-1.5pt \bigr)$.  However, the characterization
of Borel-fixed points in this more general setting introduces some challenges;
see \cite{Par96}*{Proposition~6}.  Chief among them, an algebraic description
of the Borel-fixed points is much more complicated.  For an arbitrary vector
bundle $\mathcal{E}$ on $\PP^{n}$, the scheme
$\Quot^{\kern0.5pt q} \kern-0.5pt (\mathcal{E})$ has far fewer symmetries, so
one cannot captialize on Borel-fixed points in the same way. If one goes
further and replaces $\PP^{n}$ with another projective variety---even a
Grassmannian or smooth toric variety---then as far as we know essentially
nothing is known about the geometry of the associated Quot schemes. Put
succinctly, our results mark a natural boundary for current techniques.

\section{Nonempty Quot Schemes}
\label{s:non}

\noindent
This section provides an effective numerical condition on a polynomial $q$ in
$\QQ[t]$ that characterizes when the scheme
$\Quot^{\kern0.5pt q} \kern-0.5pt(\mathcal{O}_{\PP^n}^{\!r})$ is nonempty. It
also collects some pertinent observations about integer partitions,
smoothness, Borel-fixed points, and lexicographic points.

\subsection*{Integer partitions and univariate polynomials}

A combinatorial family of polynomials plays an oversized role in our analysis.
A sequence $\lambda \coloneqq (\lambda_1, \lambda_2, \dotsc, \lambda_{e})$ of
integers is an \emph{integer partition} if
$\lambda_1 \geqslant \lambda_2 \geqslant \dotsb \geqslant \lambda_{e}
\geqslant 1$.  The nonnegative integer $e$ is the \emph{length} of $\lambda$
and each $\lambda_i$ is referred to as a \emph{part} of $\lambda$.  The empty
partition $\lambda = ()$ has length $0$.

To each integer partition
$\lambda \coloneqq (\lambda_1, \lambda_2, \dotsc, \lambda_{e})$, we associate
the integer-valued polynomial
\[ 
    p_{\lambda\!}(t)
    \coloneqq \sum_{i=1}^{e} \binom{t+\lambda_i-i}{\lambda_i-1} 
    \in \QQ[t] \, .
\]
Each summand in the definition of the polynomial $p_{\lambda}$ depends on the
\emph{size} $\lambda_i$ of the part, and the \emph{position} $i$ of the part
within the integer partition $\lambda$.  For any integer partitions $\lambda$
and $\mu$, the equality $p_{\lambda} = p_{\mu}$ of polynomials is equivalent
to the equality $\lambda = \mu$; compare with \cite{BH93}*{Lemma~4.2.7}.

Both \cite{Bre98}*{Theorem~3.5.i} and \cite{Del16}*{Lemma~3.4} prove that the
set of all polynomials associated to integer partitions is closed under
addition.  The next lemma shows that some of these polynomials cannot be
expressed as a nontrivial sum.

\begin{lemma}
    \label{l:sum} 
    For any triple $\lambda$, $\mu$, and $\nu$ of nonempty integer partitions
    satisfying $p_{\lambda} = p_{\mu} + p_{\nu}$, the integer partition
    $\lambda$ must have at least one part of size $1$.  When $\lambda$ has
    exactly one part of size $1$, there are two distinct possibilities:
    \begin{compactitem}[\phantom{W}$\bullet$]
    \item $\lambda = (\lambda_{1},2,1)$, $\mu = (\lambda_{1})$, $\nu = (2)$,
      and $\lambda_1 \geqslant 2$; or
    \item $\lambda = (\lambda_{1}, \lambda_{2}, \dotsc, \lambda_{e-1}, 1)$,
      $\mu = (\lambda_{1}, \lambda_{2}, \dotsc, \lambda_{e-1})$, $\nu = (1)$,
      and $\lambda_{e-1} \geqslant 2$.
    \end{compactitem}
\end{lemma}

\begin{proof}
  We begin by describing a recursive algorithm that outputs the integer
  partition $\lambda$ given the integer partitions $\mu$ and $\nu$ as input.
  To express the polynomial $p_{\mu} + p_{\nu}$ as a suitable sum of binomial
  coefficients, we first list the summands in the defining expressions for
  $p_{\mu}$ and $p_{\nu}$ in such a way that the denominators are
  nonincreasing and, among those with the same denominator, the numerators are
  also nonincreasing.  Each term in this list has the form
  $\smash{\binom{t+d-j}{d-1}}$, where $d$ and $j$ are positive integers.  By
  construction, there are at least $j-1$ terms listed before the term
  $\smash{\binom{t+d-j}{d-1}}$.  To proceed, we select the first term
  $\smash{\binom{t+d-j}{d-1}}$ in the list that has more than $j-1$ terms
  listed before it.  We now apply the parallel sum identity for binomial
  coefficients:
  \begin{align*}
    \binom{t + d -j}{d -1}
    &= \sum_{\ell=0}^{d -1} \binom{t-j + \ell}{\ell} 
      = \sum_{k=0}^{d -1} \binom{t+(d-k) - (j+1)}{(d-k) -1} \, . 
  \end{align*}
  This identity allows us to replace the term $\smash{\binom{t+d-j}{d-1}}$
  with $\smash{\binom{t+d-(j+1)}{d-1}}$, while simultaneously appending $d-1$
  new terms with denominators $d-2, d-3, \dotsc, 0$ to the list.  Iterating
  this process, we eventually obtain a list where, for every positive integer
  $j$, the term of the form $\smash{\binom{t+d-j}{d-1}}$ appears in position
  $j$. The sum of all the terms in this final list gives the defining
  expression for the polynomial $p_{\lambda}$.

  By symmetry, we may assume that $\mu_1 \geqslant \nu_1$.  Since $\mu$ and
  $\nu$ are nonempty, one must apply the parallel sum identity to at least
  $\smash{\binom{t + \nu_1 -1}{\nu_1-1}}$ in order to find $\lambda$.  This
  ensures that $\lambda$ has at least one part of size $1$.  When $\lambda$
  has exactly one part of size $1$, either $\nu = (1)$ or the parallel sum can
  only be applied once, as every term with zero denominator corresponds to a
  part of size $1$ in $\lambda$. When $\nu = (1)$, we see that
  $\lambda = (\lambda_{1}, \lambda_{2}, \dotsc, \lambda_{e-1}, 1)$ and
  $\mu = (\lambda_{1}, \lambda_{2}, \dotsc, \lambda_{e-1})$, where
  $\lambda_{e-1} \geqslant 2$. In the other case, $\nu_1$ must be the second
  part of $\lambda$, so there exists an integer $\lambda_{1} \geqslant 2$ such
  that $\mu = (\lambda_{1})$.  Moreover, as at most one term is appended, we
  deduce that $\nu_1 \leqslant 2$.  From the equations
  \begin{align*}
    \binom{t+(2-1) - 2}{(2)-2} 
    &= \binom{t}{0}
      = 1
      = \binom{t+(1) - 3}{(1)-1} 
      = \binom{t+(1) - 2}{(1)-1} \, , 
  \end{align*}
  we conclude that  $\lambda = (\lambda_{1},2,1)$, $\mu = (\lambda_{1})$, $\nu = (2)$, and $\lambda_{1} \geqslant 2$.
\end{proof}

\begin{example}
  \label{e:alg}
  When $\mu = (2)$ and $\nu = (2)$, we see that $\lambda = (2,2,1)$ because
  \begin{align*}
    p_{(2)} + p_{(2)}
    &= \binom{t+2-1}{2-1} + \binom{t+2-1}{2-1} \\
    &= 2t+2 
      = \binom{t+2-1}{2-1} + \binom{t+2-2}{2-1} + \binom{t+1-3}{1-1}
      = p_{(2,2,1)} \, . \qedhere
  \end{align*}
\end{example}

The algorithm in the proof of Lemma~\ref{l:sum} may be visualized as a game on
nonnegative matrices.

\begin{remark}
  As an analog to the Pascal matrix, let $\mathbf{P}$ be the matrix whose
  $(j,k)$-entry is $\binom{t+k-j}{k-1}$ for all positive integers $j$ and
  $k$. Since the entries in each row of $\mathbf{P}$ form a basis for the
  $\ZZ$-module of integer-valued polynomials in $\QQ[t]$, every integer-valued
  polynomial $q(t) \in \QQ[t]$ may be expressed as the Frobenius inner product
  of an integer matrix $\mathbf{C} \coloneqq [ c_{j,k} ]$ and $\mathbf{P}$:
  \[
    q(t) 
    = \langle \mathbf{C}, \mathbf{P} \rangle_{\text{F}}
    = \sum_{j \geqslant 1} \sum_{k \geqslant 1} c_{j,k} \binom{t+k-j}{k-1} 
    \in \QQ[t] \, . 
  \]
  From this perspective, the parallel sum identity for binomial coefficients becomes
  \[
    \binom{t+k-j}{k-1}
    = \langle \mathbf{E}_{j,k}, \mathbf{P} \rangle_{\text{F}}
    = \sum_{i=1}^{k} \langle \mathbf{E}_{j+1, i}, \mathbf{P} \rangle_{\text{F}} 
    = \sum_{i=1}^{k} \binom{t+i-(j+1)}{i-1}
    \, ,
  \]
  where the matrix unit $\mathbf{E}_{j,k}$ has only one nonzero entry with
  value $1$ in the $j$th row and $k$th column.
    
  The game starts with a nonnegative integer matrix $\mathbf{C}$ such that
  $q(t) = \langle \mathbf{C}, \mathbf{P} \rangle_{\text{F}}$. A turn involves
  using the parallel sum identity to change the matrix representing the
  polynomial $q(t)$. The game ends when the matrix has at most one nonzero
  entry in each row whose value is $1$. For instance, the matrix
  reinterpretation of Example~\ref{e:alg} is
  $2 \mathbf{E}_{1,2} \rightsquigarrow \mathbf{E}_{1,2} + \mathbf{E}_{2,2} +
  \mathbf{E}_{2,1} \rightsquigarrow \mathbf{E}_{1,2} + \mathbf{E}_{2,2} +
  \mathbf{E}_{3,1}$ or
  \[
    \begin{bmatrix}
      0 & 2 & 0  \\[-2pt]
      0 & 0 & 0  \\[-2pt]
      0 & 0 & 0  \\[-2pt]
    \end{bmatrix}
    \rightsquigarrow
    \begin{bmatrix}
      0 & 1 & 0  \\[-2pt]
      1 & 1 & 0  \\[-2pt]
      0 & 0 & 0  \\[-2pt]
    \end{bmatrix}
    \rightsquigarrow
    \begin{bmatrix}
      0 & 1 & 0 \\[-2pt]
      0 & 1 & 0 \\[-2pt]
      1 & 0 & 0  \\[-2pt]
    \end{bmatrix} \, .
  \]
\end{remark}

\begin{remark} 
  Lemma~\ref{l:sum} can also be understood via residual flags; see
  \cite{SS23}*{Definition~1.4}. Assuming that the integer partition $\lambda$
  has no part equal to $1$, every closed subscheme of $\PP^n$ with Hilbert
  polynomial $p_{\lambda}$ is a residual flag; see
  \cite{SS23}*{Theorem~3.2}. The existence of the expression
  $p_{\lambda} = p_{\mu} + p_{\nu}$ would mean that the residual flag
  associated to $p_{\lambda}$ is the disjoint union of two other schemes,
  which is impossible.
\end{remark}

\subsection*{Smoothness is local on the base}

We reduce the study of smoothness to nonsingular Quot schemes over a
field. For any positive integer $n$, let $\PP^n$ denote $n$-dimensional
projective space over a Noetherian base scheme $Y$.  Consider a locally-free
sheaf $\mathcal{E}$ on $\PP^n$. For any scheme $X$ over $Y$, set
$\mathcal{E}_X$ to be the pullback of $\mathcal{E}$ along the canonical
projection $\PP^n_X \to \PP^n$, where
$\PP^n_X \coloneqq \PP^n \mathbin{\times_{\! Y}^{}} X$ is the fibre product of
$\PP^n$ and $X$ over $Y$. Two quotients of $\mathcal{E}_X$ are equivalent if
the kernels of the quotient maps coincide as subsheaves of
$\mathcal{E}_X$. For any polynomial $q \in \QQ[t]$, the \emph{Quot scheme}
$\Quot^{\kern0.5pt q} \kern-0.5pt(\mathcal{E})$ is the projective scheme whose
$X$-valued points are equivalence classes of coherent
$\mathcal{O}_{\PP^n_{X}}$-module quotients $\mathcal{F}$ of $\mathcal{E}_{X}$,
where $\mathcal{F}$ is flat over $X$ and has Hilbert polynomial $q$ on each
fibre.  For further details, see \cite{Gro61}*{Th\'{e}or\`{e}me~3.1},
\cite{Ser06}*{Theorem~4.4.1}, or \cite{HL10}*{Theorem~2.2.4}.

We show that the smoothness of $\Quot^{\kern0.5pt q} \kern-0.5pt(\mathcal{E})$
is a local property; compare with \cite{SS23}*{Lemma~3.1}.

\begin{lemma}
  \label{l:smooth}
  Let $\PP^n$ denote $n$-dimensional projective space over a Noetherian base
  scheme $Y$ and let $\mathcal{E}$ be a locally-free
  $\mathcal{O}_{\PP^{n}}$-module of constant finite rank $r$.  For any
  polynomial $q \in \QQ[t]$, the Quot scheme
  $\Quot^{\kern0.5pt q} \kern-0.5pt(\mathcal{E})$ is smooth over $Y$ if and
  only if its fibre is nonsingular over every point of $Y$.
\end{lemma}

\begin{proof}
  The structure map $\Quot^{\kern0.5pt q} \kern-0.5pt(\mathcal{E}) \to Y$ is
  strongly projective, implying that it is flat and locally of finite
  presentation; see \cite{AK80}*{Theorem~2.6}. Hence, the structure map is
  smooth if and only if its fibres are nonsingular over every point of $Y$;
  see
  \cite{SP24}*{\href{https://stacks.math.columbia.edu/tag/01V8}{Tag~01V8}}.
\end{proof}

As a consequence, we may assume that the base scheme $Y$ is the spectrum of a
field $\kk$ when focusing on smoothness.

\subsection*{Borel-fixed points}

In the rest of the document, we assume that the ambient sheaf is the trivial
vector bundle
$\mathcal{E} = \mathcal{O}_{\PP^n}^{\! r} = \bigoplus_{\! j=1}^r
\mathcal{O}_{\PP^n}^{}$ on $\PP^n$ of rank $r$.  Our analysis of
$\Quot^{\kern0.5pt q} \kern-0.5pt (\mathcal{O}_{\PP^n}^{\! r})$ relies on a
natural group action.  To review, the homogeneous coordinate ring of $\PP^n$
is the $\ZZ$-graded polynomial ring $S \coloneqq \kk[x_0,x_1, \dotsc, x_n]$,
where $\deg(x_i) \coloneqq 1$ for all $0 \leqslant i \leqslant n$.  Let
$\bm{e}_1, \bm{e}_2, \dotsc, \bm{e}_r$ denote the standard basis of the free
$\ZZ$-graded $S$-module $S^r \coloneqq \bigoplus_{j=1}^r S$, where
$\deg(\bm{e}_{j}) \coloneqq 0$ for all $1 \leqslant j \leqslant r$.  The Serre
vanishing theorem implies that every quotient $\mathcal{F}$ of
$\mathcal{O}_{\PP^n}^{\! r}$ is determined by an $S$-module
$S^{r} \mathbin{\!\!/\!} K$, where $K$ is a homogeneous submodule of
$S^{r}$. A \emph{monomial} in $S^r$ is an element of the form
$\bm{x}^{\bm{u}} \, \bm{e}_{j}$, where
$\bm{x}^{\bm{u}} \coloneqq \smash{x_0^{u_0} x_1^{u_1} \dotsb \,
  x_{n\phantom{1}}^{u_{n\phantom{1}}}}$ is a monomial in $S$ and
$1 \leqslant j \leqslant r$. A \emph{monomial submodule} of $S^r$ is a
submodule generated by monomials, so every monomial submodule of $S^r$ is
homogeneous and can be expressed as a direct sum
$K = \bigoplus_{j=1}^{r} I_{(j)} \, \bm{e}_j$, where each $I_{(j)}$ is a
monomial ideal in the ring $S$.

For the basic group action, recall that for any nonnegative integer $m$, the
general linear group $\GL(m, \kk)$ consists of all invertible
$(m \mathbin{\times} m)$-matrices with entries in the field $\kk$.

\begin{lemma} 
  The group $\GL(n+1,\kk)\times \GL(r,\kk)$ acts on the free $\ZZ$-graded
  $S$-module $S^r$.
\end{lemma}

\begin{proof} 
  The group $\GL(n + 1, \kk)$ acts as the automorphisms of the $\ZZ$-graded
  $\kk$-algebra $S$, and this induces a compatible action on the free module
  $S^{r}$. The group $\GL(r,\kk)$ acts on the free module $S^{r}$ via its
  basis vectors. Since the actions of $\GL(n+1, \kk)$ and $\GL(r, \kk)$
  commute, the group $\GL(n+1, \kk) \times \GL(r, \kk)$ acts on the free
  $\ZZ$-graded $S$-module $S^r$.
\end{proof}

The action of a subgroup of $\GL(n+1,\kk)\times \GL(r,\kk)$ is crucial for
understanding
$\Quot^{\kern0.5pt q} \kern-0.5pt(\mathcal{O}_{\PP^{n}}^{\! r})$.  We choose
the \emph{Borel subgroup} of this product group to be
$\operatorname{B}(n+1, \kk) \times \operatorname{B}(r,\kk)$, where
$\operatorname{B}(n+1,\kk)$ and $\operatorname{B}(r,\kk)$ are the subgroups of
upper triangular matrices in $\GL(n+1, \kk)$ and $\GL(r, \kk)$
respectively. The Borel-fixed point theorem~\cite{Bor91}*{Proposition~15.2}
asserts that any nonempty Quot scheme has at least one Borel-fixed point.  In
this context, a Borel-fixed point in
$\Quot^{\kern0.5pt q} \kern-0.5pt(\mathcal{O}_{\PP^{n}}^{\! r})$ is associated
to a quotient module $S^r \mathbin{\!\!/\!} K$, where the submodule $K$ is
invariant under the action of the Borel subgroup
$\operatorname{B}(n+1, \kk) \times \operatorname{B}(r,\kk)$.  Moreover,
\cite{Par96}*{Proposition~6} shows that a submodule $K$ of $S^{r}$ is
Borel-fixed if and only if it satisfies the following three conditions:
\begin{compactenum}[\phantom{W} (B1)]
\item
  $K = I_{(1)} \, \bm{e}_1 \oplus I_{(2)} \, \bm{e}_2 \oplus \dotsb \oplus
  I_{(r)} \, \bm{e}_r$, where each $I_{(j)}$ is a monomial ideal in $S$,
\item for any $1 \leqslant j \leqslant r$, each ideal $I_{(j)}$ is invariant
  under the action of the Borel subgroup, and
\item the sequence
  $I_{(1)} \supseteq I_{(2)} \supseteq \dotsb \supseteq I_{(r)}$ of ideals
  form a descending chain.
\end{compactenum}

\begin{remark} 
  Definition~7 in \cite{Par96} declares that a submodule is \emph{standard}
  Borel-fixed if it satisfies conditions~(B1) and (B3), and, for any monomial
  $\bm{x}^{\bm{u}} \bm{e}_{\kern-0.5pt j}$ in $K$ and any variable $x_{i}$
  that divides $\bm{x}^{\bm{u}}$, the monomial
  $x_{k}^{} \, x_{i}^{-1} \bm{x}^{\bm{u}} \bm{e}_{\kern-0.5pt j}$ belongs to
  $K$ for all $k < i$.  This combinatorial strengthening of condition~(B2) is
  equivalent to asserting that each monomial ideal $I_{(j)}$ is \emph{strongly
    stable}.  Every standard Borel-fixed submodule is Borel-fixed.  If the
  characteristic of the field $\kk$ is $0$, then every Borel-fixed submodule
  is standard.  However, when the characteristic of $\kk$ is a positive prime
  $p$, the $S$-ideal $\langle x_0^{p}, x_1^{p} \rangle$ is a nonstandard
  Borel-fixed ideal.  Fortuitously, we only need standard Borel-fixed
  submodules in our analysis of the Quot schemes
  $\Quot^{\kern0.5pt q} \kern-0.5pt(\mathcal{O}_{\PP^{n}}^{\! r})$, thereby
  producing uniform results over the integers.
\end{remark}

\subsection*{Lexicographic points} 

We identify a distinguished Borel-fixed point on each nonempty
$\Quot^{\kern0.5pt q} \kern-0.5pt(\mathcal{O}_{\PP^{n}}^{\! r})$.  To specify
this point, we utilize a monomial order on the free $S$-module $S^{r}$.  The
\emph{lexicographic order} on monomials $\bm{x}^{\bm{u}}$ and
$\bm{x}^{\bm{v}}$ in $S$ is defined by $\bm{x}^{\bm{u}} > \bm{x}^{\bm{v}}$ if
the first nonzero entry in the integer sequence
$(u_0-v_0, u_1 -v_1, \dotsc, u_n-v_n)$ is positive. The \emph{lexicographic
  order} on monomials $\bm{x}^{\bm{u}} \, \bm{e}_{i}$ and
$\bm{x}^{\bm{v}} \, \bm{e}_{\kern-0.5pt j}$ in $S^{r}$ is defined by
$\bm{x}^{\bm{u}} \, \bm{e}_{i} > \bm{x}^{\bm{v}} \, \bm{e}_{\kern-0.5pt j}$ if
$i < j$ or $i = j$ and $\bm{x}^{\bm{u}} > \bm{x}^{\bm{v}}$. For any
nonnegative integers $d$ and $e$, the \emph{lexicographic subspace} of
$(S^r)_d$ of dimension $e$ is the linear subspace spanned by the first $e$
monomials of $(S^r)_d$ ordered lexicographically.  A submodule $K$ of $S^r$ is
called \emph{lexicographic} if it is homogeneous and, for each degree $d$, the
homogeneous component $K_d$ is a lexicographic subspace of $(S^r)_d$.  It
follows that any lexicographic submodule is a standard Borel-fixed submodule.

Building on Macaulay's characterization of Hilbert polynomials for homogeneous
ideals in $S$, we record the following criterion; see
\cite{BdC88}*{Th{\'e}or{\`e}me~2} and compare with \cite{SS23}*{Lemma~2.3}.

\begin{proposition}
  \label{p:nonempty}
  For any polynomial $q \in \QQ[t]$, the following are equivalent:
  \begin{compactenum}[\upshape (a)]
  \item There exists a surjective morphism
    $\mathcal{O}_{\PP^{n}}^{\! r} \to \mathcal{F}$ of
    $\,\mathcal{O}_{\PP^{n}}$-modules such that
    $\mathcal{F} \neq \mathcal{O}_{\PP^{n}}^{\! r}$ and the Hilbert polynomial
    of $\,\mathcal{F}$ is
    $q(t) = \chi\bigl( \mathcal{F}(t) \kern-1.5pt \bigr) \kern-1.0pt \coloneqq
    \sum_{i} (-1)^{i} \dim_{\kk} H^i \kern-0.5pt \bigl( \PP^{n},
    \mathcal{F}(t) \kern-1.0pt \bigr)$.
  \item There exists a nonnegative integer $s$ and an integer partition
    $\lambda$ such that $r > s$, $n \geqslant \lambda_1$, and
    \[
      q(t) 
      = s \binom{t+n}{n} + p_{\lambda}(t) 
      = s \binom{t+n}{n} + \sum_{i=1}^{e} \binom{t+\lambda_i-i}{\lambda_i-1} 
      \, . 
    \]
  \item There exists a nonzero lexicographic submodule $K$ of $S^r$ such that
    the Hilbert function of the quotient $S^r \mathbin{\!\!/\!} K$ equals
    $q(j)$ for any nonnegative integer $j$.
  \end{compactenum}
\end{proposition}

The nonnegative integer $s$ is the rank of the coherent sheaf $\mathcal{F}$;
see \cite{HL10}*{Definition~1.2.2}. In the edge case
$\mathcal{F} = \mathcal{O}_{\PP^{n}}^{\! r}$, the lexicographic submodule is
$K = 0$ and the Hilbert polynomial is $r \binom{t+n}{n}$, so $s = r$ and
$\lambda$ is the empty integer partition.

\begin{proof}[Sketch of proof]
  As already observed in \cite{Hul95}*{Corollary~6} and
  \cite{Par96}*{Proposition~2}, this is a relatively straightforward extension
  of the ideas in \cite{Mac27}.
    \begin{compactitem}[$\bullet$]
    \item[$(a) \!\Rightarrow\!\! (b)$:] Every quotient $\mathcal{F}\!$ of
      $\mathcal{O}_{\PP^{n}}^{\! r}$ is determined by an $S$-module
      $S^r \mathbin{\!\!/\!} K$ for some homogeneous submodule $K$ in $S^r$.
      The submodule $\operatorname{in}(K)$, consisting of the initial terms of
      elements in $K$ with respect to the lexicographic order, is a monomial
      submodule of $S^r$ with the same Hilbert function as $K$.  Hence, there
      exist monomial ideals $I_{(1)}, I_{(2)}, \dotsc, I_{(r)}$ in $S$ such
      that
      $\operatorname{in}(K) = I_{(1)} \, \bm{e}_1 \oplus I_{(2)} \, \bm{e}_{2}
      \oplus \dotsb \oplus I_{(r)} \, \bm{e}_{r}$.  The Hilbert polynomial $q$
      of $\mathcal{F}$ (and of $S^r \mathbin{\!\!/\!} K$) is, therefore, the
      sum of the Hilbert polynomials for the quotients
      $S^1 \mathbin{\!\!/\!} I_{(i)}$.  The Macaulay theorem implies that, for
      each nonzero ideal $I_{(i)}$, there exists an integer partition
      $\lambda \kern-1.0pt (i)$ such that
      $\lambda \kern-1.0pt (i)_1 \leqslant n$ and the Hilbert polynomial of
      $S^1 \mathbin{\!\!/\!} I_{(i)}$ is $p_{\lambda \kern-1.0pt (i)}$.
      Furthermore, there is an integer partition $\lambda$ such that
      $\lambda_1 \leqslant n$ and the associated polynomial $p_{\lambda}$ is
      the sum of the $p_{\lambda(i)}$, where $I_{(i)} \neq 0$.  Let $s$ be the
      number of ideals $I_{(1)}, I_{(2)}, \dotsc, I_{(r)}$ that are equal to
      $0$.  Since the Hilbert polynomial of $S$ is $\smash{\binom{t+n}{n}}$,
      we deduce that the Hilbert polynomial $q$ of $\mathcal{F}$ is
      $s \, \binom{t+n}{n} + p_{\lambda\!}(t)$.
    \item[$(b) \!\Rightarrow\!\! (c)$:] Given an integer partition $\lambda$
      satisfying $\lambda_1 \leqslant n$, the associated polynomial
      $p_{\lambda}$ satisfies the growth condition in Macaulay's
      characterization; see \cite{BH93}*{Theorem~4.2.10}.  Hence, there exists
      a unique lexicographic ideal $L(\lambda)$ in $S$ such that the Hilbert
      function of $S^1 \mathbin{\!\!/\!} L(\lambda)$ is $p_{\lambda}$.  It
      follows that the monomial submodule
      \[
        \phantom{w} K = \langle 1_{S} \rangle \bm{e}_{1} \oplus \langle 1_{S}
        \rangle \bm{e}_{2} \oplus \dotsb \oplus \langle 1_{S} \rangle
        \bm{e}_{r-s-1} \oplus L(\lambda) \bm{e}_{r-s} \oplus \langle 0_{S}
        \rangle \bm{e}_{r-s+1} \oplus \langle 0_{S} \rangle \bm{e}_{r-s+2}
        \oplus \dotsb \oplus \langle 0_{S} \rangle \bm{e}_{r}
      \]
      is lexicographic and the Hilbert function of the quotient
      $S^r \mathbin{\!\!/\!} K \cong S^1 \mathbin{\!\!/\!} L(\lambda) \oplus
      S^{s}$ is $s \, \binom{t+n}{n} + p_{\lambda}$.
    \item[$(c) \!\Rightarrow\!\! (a)$:] Take $\mathcal{F}$ to be the sheaf
      associated to the quotient $S$-module $S^1 \mathbin{\!\!/\!}
      K$. \qedhere
    \end{compactitem}
\end{proof}

The next examples illustrate how the pair $(s, \lambda)$ give rise to a Quot
scheme.

\begin{example}
  \label{e:1pt}
  When $s = 0$ and $\lambda = (1)$, the associated polynomial is
  $p_{\lambda\!}(t) = 1$.  For all positive integers $n$ and $r$, we have
  $\Quot^{1} \! ( \mathcal{O}_{\PP^n}^{\! r}) = \PP(\mathcal{O}_{\PP^n}^{\!
    r}) \cong \PP^{r-1} \mathbin{\times} \PP^n$; see
  \cite{Kle90}*{Proposition~2.2}.
\end{example}

\begin{example}
  \label{e:TP}
  From the Euler sequence
  \[
    0 
    \longrightarrow \mathcal{O}_{\PP^{n}}
    \longrightarrow \mathcal{O}_{\PP^{n}}^{n+1}(1)
    \longrightarrow \mathcal{T}_{\PP^{n}} 
    \longrightarrow 0 \, ,
  \]
  we see that the Hilbert polynomial of the twisted tangent bundle
  $\mathcal{T}_{\PP^{n}}(-1)$ is
  \[
    q(t)
    \coloneqq (n+1) \binom{t+n}{n} - \binom{t-1+n}{n} 
    = n \binom{t+n}{n} + \binom{t+n-1}{n-1} \, .
  \]
  In particular, when $r = n + 1$, $s = n$, and $\lambda = (n)$, the Quot
  scheme $\Quot^{\kern0.5pt q} \kern-0.5pt (\mathcal{O}_{\PP^n}^{\! n+1})$ is
  nonempty.
\end{example}

Since the leading coefficient of each summand in the expression for the
rational univariate polynomial $q$ in Proposition~\ref{p:nonempty} is
nonnegative, one can use a greedy algorithm to determine whether $q$ indexes a
nonempty Quot scheme.

\begin{example}
  Assume that $n = 2$ and $r \geqslant 2$, or $n \geqslant 3$ and
  $r \geqslant 1$.  Since
  \[
    q(t) 
    \coloneqq t^2 
    = 2 \binom{t+2}{2} - (3t+2)
    = \binom{t+3-1}{3-1} + \binom{t+3-2}{3-1} - (2t+1)  \, ,
  \] 
  Proposition~\ref{p:nonempty} implies that
  $\Quot^{\kern0.5pt q} \kern-0.5pt(\mathcal{O}_{\PP^n}^{\! r})$ is empty;
  compare with \cite{Bre98}*{Theorem~3.8(i)}.
\end{example}

\begin{example}
  Assume that $n \geqslant 2$ and $r \geqslant 1$.  Since 
  \[
    q(t) 
    \coloneqq \binom{t}{n-1} 
    = \binom{t+n-1}{n-1} 
    - \left( \frac{(n-1)}{(n-2)!} t^{n-1} + \dotsb + 1 \right) \, ,
  \]  
  Proposition~\ref{p:nonempty} implies that
  $\Quot^{\kern0.5pt q} \kern-0.5pt(\mathcal{O}_{\PP^n}^{\! r})$ is empty;
  compare with \cite{Bre98}*{Theorem~3.8(iv)}.
\end{example}

The lexicographic point on the Quot scheme
$\Quot^{\kern0.5pt q} \kern-0.5pt(\mathcal{O}_{\PP^{n}}^{\! r})$ admits a
concrete algebraic description.  To present the lexicographic submodule
$S^{r-s-1} \oplus L(\lambda) \oplus 0^{s} \subseteq S^r$ corresponding to the
integer partition $\lambda$, it is enough to specify the monomial generators
of the saturated ideal $L(\lambda)$ in $S$.

\begin{remark}
  \label{r:gen}
  For any positive integer $j$, let $a_j$ denote the number of parts in the
  integer partition $\lambda$ equal to $j$; the nonnegative integer $a_j$ is
  called the \emph{multiplicity} of the part $j$ in $\lambda$. Using
  superscripts to indicate the number of repetitions of a part leads to the
  alternative notation
  \[
    \lambda 
    = (n^{a_{n}}, \dotsc, j^{\kern0.75pt a_{\kern-0.5pt j}}, \dotsc, 2^{a_{2}}, 1^{a_{1}}) \, .
  \]
  Assuming that $n \geqslant \lambda_1$, the multiplicities
  $a_1, a_2, \dotsc, a_n$ and the length $e$ of the integer partition
  $\lambda$ are related by the equation $e = a_1 + a_2 + \dotsb + a_n$.

  The \emph{lexicographic ideal} in the polynomial ring $S \coloneqq \kk[x_0,x_1, \dotsc, x_n]$ is
  \[
    L(\lambda) \coloneqq \bigl\langle
      x_0^{a_{n\vphantom{-1}}+1}, \,
      x_{0}^{a_{n\vphantom{-1}} \vphantom{+1}} \, x_{1}^{a_{n-1}+1}, \, \dotsc, \, 
      x_{0}^{a_{n\vphantom{-1}}} \, x_{1}^{a_{n-1}} \!\dotsb x_{n-3}^{a_{3\vphantom{-1}}} \, 
      x_{n-2}^{a_{2\vphantom{-1}}+1}, \, 
      x_{0}^{a_{n\vphantom{-1}}\vphantom{+1}} \, x_{1}^{a_{n-1}} \!\dotsb x_{n-2}^{a_{2\vphantom{-1}}} \, 
      x_{n-1}^{a_{1\vphantom{-1}}} 
    \bigr\rangle \, ;
  \]
  compare with \cite{RS97}*{Notation~1.2} and
  \cite{SS23}*{Proposition~2.5}. These $n$ monomials do not always form a
  minimal generating set for the lexicographic ideal.  When $a_1 > 0$, the
  given monomials are incomparable (in the distributive lattice ordered by
  divisibility) and, thereby, form a minimal set of generators.  When
  $a_1 = a_2 = \dotsb = a_{k-1} = 0$ and $a_k \neq 0$ for some positive
  integer $k$, the penultimate $k-1$ monomials are divisible by the final one,
  while the remaining monomials are not.  Hence, setting $k$ to be the
  smallest part in the integer partition $\lambda$, the ideal $L(\lambda)$ has
  $n - (k-1) = n-k+1$ minimal generators.  The projective dimension of the
  $\ZZ$-graded $S$\nobreakdash-module $S^1 \mathbin{\!\!/\!}  L(\lambda)$ is
  also equal to $n-k+1$ because its minimal free resolution is given
  explicitly by the Eliahou--Kervaire resolution; see
  \cite{PS08}*{Theorem~2.3}.  Alternatively, one may compute the depth of
  ideal $L(\lambda)$ from its unique irredundant irreducible
  decomposition~\cite{SS23}*{Proposition~2.5} and use the Auslander--Buchsbaum
  formula to find its projective dimension.
\end{remark}

Relatively simple conditions on the integer partition guarantee that the
lexicographic submodule is the only saturated Borel-fixed submodule; compare
with \cite{Sta20}*{Theorem~1.1}.

\begin{corollary} 
  \label{c:!b}
  Let $r$ and $s$ be nonnegative integers satisfying $r > s$. For every
  integer partition
  $\lambda \coloneqq (\lambda_{1}, \lambda_{2}, \dotsc, \lambda_{e})$ such
  that either $\lambda_{e} \geqslant 2$, $\lambda = (1)$, or $\lambda = ()$,
  there exists a unique saturated Borel-fixed submodule $K \subseteq S^r$ such
  that the Hilbert polynomial of $S^r/K$ is
  $q(t) = s \binom{t+n}{n} + p_{\lambda}(t) \in \QQ[t]$.  Moreover, when
  $r \geqslant s+2$, the converse holds.
\end{corollary}

\begin{proof} 
  Proposition~\ref{p:nonempty} shows that the lexicographic submodule
  $S^{r-s-1} \oplus L(\lambda) \oplus 0^{s} \subseteq S^r$ has Hilbert
  polynomial $q$. When $r-s \geqslant 2$, there is another Borel-fixed
  submodule if $p_{\lambda} = p_{\mu} + p_{\nu}$ for some nonempty integer
  partitions $\mu$ and $\nu$: the submodule is
  $S^{r-s-2} \oplus L(\nu) \oplus L(\mu) \oplus 0^{s} \subseteq S^r$. Since
  $\lambda_{e} \geqslant 2$, $\lambda = (1)$, or $\lambda = ()$,
  Lemma~\ref{l:sum} precludes the existence of such a nontrivial sum.
  Conversely, for any integer partition
  $\lambda = (\lambda_1, \lambda_2, \dotsc, \lambda_{e-1},1)$ having at least
  one part of size $1$, we always have $p_{\lambda} = p_{\mu} + p_{\nu}$,
  where $\mu = (\lambda_{1}, \lambda_{2}, \dotsc, \lambda_{e-1})$ and
  $\nu = (1)$.

  Given the conditions on $\lambda$, it follows that any Borel-fixed submodule
  with Hilbert polynomial $q$ is a direct sum of monomial ideals in which at
  most one summand is a nonzero proper ideal. However, the hypotheses on
  $\lambda$ also ensure that the lexicographic ideal $L(\lambda)$ is the
  unique saturated Borel-fixed ideal with Hilbert polynomial $p_{\lambda}$;
  see \cite{Sta20}*{Theorem~1.1} or \cite{SS23}*{Theorem~3.2}.
\end{proof}

\section{Smoothness Of The Lexicographic Point}
\label{s:sm}

\noindent 
In this section, we demonstrate that the lexicographic point on
$\Quot^{\kern0.5pt q} \kern-0.5pt (\mathcal{O}_{\PP^n}^{\! r})$ is always
smooth. Throughout, the base scheme is $\Spec(\kk)$ and, for any integer
partition $\lambda$, we assume that $n \geqslant \lambda_1$.  The homogeneous
coordinate ring of the $n$-dimensional projective space $\PP^n$ is still
denoted by $S \coloneqq \kk[x_0, x_1, \dotsc, x_n]$. For any integer $k$ and
any finitely generated $\ZZ$-graded $S$-module $M$, the graded piece $M_k$ is
the $\kk$-vector space spanned by all elements in $M$ of degree $k$.  The
$\ZZ$-graded $S$-module $M$ determines the coherent sheaf $\widetilde{M}$ on
$\PP^n$. For any integer $i$, set
\[
  h^i(M)
  \coloneqq \dim_{\kk} H^i(\PP^n \kern-1.0pt, \kern2.0pt \widetilde{M}) \, .
\]  

For convenience, we recall the following observation.

\begin{lemma}[\cite{RS97}*{Lemma~3.1}]
  \label{l:h^i}
  Let $M$ be a finitely generated $\ZZ$-graded $S$-module. For any
  nonnegative integer $i$, we have
  \[
    h^i({M}) 
    = \begin{cases} 
      \dim_\kk M_0 + \dim_\kk \Ext^{n}_S(M, S)_{-n-1} + \dim_\kk \Ext^{n+1}_S(M, S)_{-n-1} 
      & \text{if $i = 0$,} \\[-2pt]
      \dim_\kk \Ext^{n-i}_S(M, S)_{-n-1} 
      & \text{if $i > 0$.} \\[-1pt]
    \end{cases}
  \]  
\end{lemma}

Using this formula, we can compute the number of linearly independent global
sections of the structure sheaf for the closed subscheme of $\PP^n$
corresponding to a lexicographic ideal $L(\lambda)$ in $S$.  Curiously, this
dimension depends on the multiplicity $a_1$ of the part $1$ in the integer
partition $\lambda$.

\begin{lemma}
  \label{l:h^0}
  For any integer partition
  $\lambda \coloneqq (n^{a_n}, \dotsc, 2^{a_{2}}, 1^{a_{1}})$ of length
  $e \coloneqq a_1 + a_2 + \dotsb + a_n > 0$, we have
  \[ 
    h^0 \kern-0.5pt \bigl( S^1 \mathbin{\!\!/\!} L(\lambda) \kern-1.5pt \bigr) 
    = \begin{cases} 
      a_{1}   & \text{if $e = a_{1}$,} \\[-2pt]
      a_{1}+1 & \text{if $e > a_{1}$.} \\[-1pt]
    \end{cases}
  \]
\end{lemma}

\begin{proof}
  Suppose that $e = a_1$.  Since $\lambda = (1^{a_1})$, the associated Hilbert
  polynomial is the constant $p_{\lambda}(t) = a_1$, so the closed subscheme
  of $\PP^n$ corresponding to the lexicographic ideal $L(\lambda)$ in $S$ is
  $0$-dimensional and affine.  We deduce that
  $h^0 \kern-0.5pt \bigl( S^1 \mathbin{\!\!/\!} L(\lambda) \kern-1.5pt \bigr)
  = a_1$.

  Suppose that $e > a_1$, which implies that
  $n \geqslant \lambda_1 \geqslant \lambda_{e-a_1} > 1$.  We proceed by
  induction on $a_1$.  When $a_1 = 0$, Remark~\ref{r:gen} shows that no
  generator of the monomial ideal $L(\lambda)$ involves the variables
  $x_{n-1}$ or $x_{n}$ in $S$. Thus, the projective dimension of the quotient
  $S^1 \mathbin{\!\!/\!} L(\lambda)$ is less than $n$ and Lemma~\ref{l:h^i}
  shows that
  $h^0 \kern-0.5pt \bigl( S^1 \mathbin{\!\!/\!} L(\lambda) \kern-1.5pt \bigr)
  = \dim_{\kk} \smash{\kern-0.5pt \bigl( S^1 \mathbin{\!\!/\!} L(\lambda)
    \kern-1.5pt \bigr)}_{0}^{} = 1$, completing the base case.  Now, assume
  that $a_1 > 0$. Consider the subpartition
  $\mu \coloneqq (n^{a_n}, \dotsc, 2^{a_{2}}, 1^{a_{1}-1})$ having length
  $e-1$. From the monomial generators in Remark~\ref{r:gen}, we obtain the
  proper inclusion $L(\lambda) \subset L(\mu)$ and the isomorphism
  $P \coloneqq L(\mu) \mathbin{\!\!/\!} L(\lambda) \cong S^1(-e+1)
  \mathbin{\!\!/\!}  \langle x_0, x_1, \dotsc, x_{n-1} \rangle$. In
  particular, the coherent sheaf $\widetilde{P}$ is (up to a twist) the
  structure sheaf of a single point on $\PP^n$, so $h^0(P) = 1$. Hence, the
  exact sequence
  \[  
    0 \longrightarrow
    P \longrightarrow
    S^1 \mathbin{\!\!/\!} L(\lambda) \longrightarrow
    S^1 \mathbin{\!\!/\!} L(\mu) \longrightarrow
    0 \, ,
  \]
  of $\ZZ$-graded $S$-modules yields the equation
  $h^0 \kern-0.5pt \bigl( S^1 \mathbin{\!\!/\!} L(\lambda) \kern-1.5pt \bigr)
  = h^0 \kern-0.5pt \bigl( S^1 \mathbin{\!\!/\!} L(\mu) \kern-1.5pt \bigr) +
  1$.  From the induction hypothesis, we conclude that
  $h^0 \kern-0.5pt \bigl( S^1 \mathbin{\!\!/\!} L(\lambda) \kern-1.5pt \bigr)
  = a_1+1$.
\end{proof}

Dualizing the lexicographic ideal, we see that the number of linearly
independent global sections depends on the multiplicity $a_n$ of the part $n$
in the integer partition $\lambda$.

\begin{lemma}
  \label{l:dual} 
  For any integer partition
  $\lambda \coloneqq (n^{a_n}, \dotsc, 2^{a_{2}}, 1^{a_{1}})$, we have
  \[
    h^0 \kern-0.5pt \Bigl( \Hom_S \bigl( L(\lambda), S \bigr) \kern-2.5pt
    \Bigr) = \binom{n+a_n}{n} \, .
  \]
\end{lemma}

\begin{proof} 
  Since the dual of a monomial ideal is equal to the intersection of the
  principal fractional ideals generated by the inverses of its minimal
  generators, Remark~\ref{r:gen} implies that
  \[
    \Hom_S \bigl( L(\lambda), S \bigr) 
    = \left( x_{0}^{-a_n-1} \cdot S \right) 
    \cap \left( x_{0}^{-a_n} x_{1}^{-a_{n-1}-1} \cdot S \right) \cap \dotsb
    \cap \left( x_{0}^{-a_n} x_{1}^{-a_{n-1}} \dotsb x_{n-2}^{-a_2} x_{n-1}^{a_{1}} \cdot S \right)
    \cong S(a_{n}) \, . 
  \]
  Thus, Lemma~\ref{l:h^i} gives
  $\smash{h^0 \kern-0.5pt \Bigl( \kern-0.5pt \Hom_S \bigl( L(\lambda), S
    \bigr) \kern-3.0pt \Bigr)} = h^0 \bigl( S(a_{n}) \kern-1.0pt \bigr) =
  \dim_{\kk} S_{a_{n}} = \binom{n+a_{n}}{n}$.
\end{proof}

Harnessing these lemmas, we determine the dimension of the tangent space to
$\Quot^{\kern0.5pt q}\kern-0.5pt (\mathcal{O}_{\PP^n}^{\! r})$ at the
lexicographic point. For our purposes, it is enough to compare this dimension
with that of the tangent space to the Hilbert scheme
$\Hilb^{\kern0.5pt p_\lambda} \kern-0.5pt (\PP^n)$ at the lexicographic
point. \cite{RS97}*{Theorem~1.4} proves that the lexicographic point in
$\Hilb^{\kern0.5pt p_\lambda} \kern-0.5pt (\PP^n)$ is smooth, so it lies on a
unique irreducible component, called the \emph{lexicographic component}.
Moreover, \cite{RS97}*{Theorem~4.1} provides a formula for the dimension of
this component; also see \cite{SS23}*{Remark~2.7}.

\begin{proposition}
  \label{p:dim}
  Fix a nonnegative integer $s$ and an integer partition
  $\lambda \coloneqq (\lambda_{1}, \lambda_{2}, \dotsc, \lambda_{e})$ such
  that $r > s$, $n \geqslant \lambda_{1}$, and
  $q(t) \coloneqq s \binom{t+n}{n} + p_{\lambda}(t) = s \binom{t+n}{n} +
  \sum_{i=1}^{e} \binom{t+\lambda_i-i}{\lambda_i-1}$.  Let $N_{\lambda}$
  denote the dimension of the lexicographic component of the Hilbert scheme
  $\Hilb^{\kern0.5pt p_\lambda} \kern-0.5pt (\PP^n)$. When
  $\lambda = (n^{a_n}, \dotsc, 2^{a_{2}}, 1^{a_{1}})$, the dimension of the
  tangent space to the Quot scheme
  $\Quot^{\kern0.5pt q}\kern-0.5pt (\mathcal{O}_{\PP^n}^{\! r})$ at the
  lexicographic point is
  \[  
    \begin{cases}
      N_{\lambda} + s \binom{n+a_n}{n} + (r-s-1)(s+a_1) 
      & \text{if $e = a_1$,} \\[2pt]
      N_{\lambda} + s \binom{n+a_n}{n} + (r-s-1)(s+a_1+1) 
      & \text{if $e > a_1$.} \\[-1pt]
    \end{cases}
  \]
\end{proposition}

\begin{proof} 
  The lexicographic point in the Quot scheme
  $\Quot^{\kern0.5pt q} \kern-0.5pt (\mathcal{O}_{\PP^n}^{\! r})$ corresponds
  to the submodule $S^{r-s-1} \oplus L(\lambda) \oplus 0^{s} \subseteq S^{r}$,
  where $L(\lambda)$ is the lexicographic ideal associated to the integer
  partition $\lambda$. Hence, the tangent space to this Quot scheme at this
  point is the space of the global sections of the sheaf associated to the
  $\ZZ$-graded $S$-module
  $\Hom_{S} \kern-0.5pt \bigl( S^{r-s-1} \oplus L(\lambda), \, S^1
  \mathbin{\!\!/\!} L(\lambda) \oplus S^{s} \kern-1.0pt \bigr)$; see
  \cite{Gro61}*{Corollaire~5.3} or \cite{Ser06}*{Corollary~4.4.6}. Similarly,
  the tangent space to $\Hilb^{\kern0.5pt p_\lambda} \kern-0.5pt (\PP^n)$ at
  the lexicographic points is the space of global sections of the sheaf
  associated to $\ZZ$-graded $S$-module
  $M \coloneqq \Hom_{S} \kern-0.5pt \bigl( L(\lambda), S^1 \mathbin{\!\!/\!}
  L(\lambda) \kern-1.5pt \bigr)$, so $N_{\lambda} = h^{0}(M)$. Since
  \[
    \Hom_{S} \kern-0.5pt \bigl( S^{r-s-1} \oplus L(\lambda), 
    \, S^1 \mathbin{\!\!/\!} L(\lambda) \oplus S^{s} \kern-1.0pt \bigr)
    \cong S^{(r-s-1)s} 
    \oplus \left(
      \bigoplus_{i=1}^{r-s-1} S^{1} \mathbin{\!\!/\!} L(\lambda) \kern-4.0pt
    \right)
    \oplus \left( \bigoplus_{j=1}^{s} \Hom_S \bigl( L(\lambda), S \bigr) 
      \kern-4.0pt \right) \oplus M \, ,
  \]
  the dimension of the desired tangent space is
  \[
    (r-s-1)s 
    + (r-s-1) \, h^0 \bigl(
      S^{1} \mathbin{\!\!/\!} L(\lambda) \kern-1.0pt
    \bigr)
    + s \, \smash{h^0 \kern-0.5pt \Bigl(
      \kern-0.5pt \Hom_S \bigl( L(\lambda), S \bigr)
      \kern-3.0pt \Bigr)} 
    + h^{0}(M) \, . 
  \]
  Using Lemma~\ref{l:h^0} and Lemma~\ref{l:dual} establishes the formula.
\end{proof}

Having determined the dimension of the tangent space at the lexicographic
point, we now construct a family of the same dimension passing through that
point. We begin with a special case.

\begin{lemma} 
  \label{l:s+1}
  Let $s$ be a nonnegative integer, let $\lambda$ be an integer partition
  $\lambda \coloneqq (\lambda_{1}, \lambda_{2}, \dotsc, \lambda_{e})$ such
  that $n \geqslant \lambda_{1}$, and consider the polynomial
  $q(t) \coloneqq s \binom{t+n}{n} + p_{\lambda}(t) \in \mathbb{Q}[t]$. When
  $r = s+1$, the lexicographic point in
  $\Quot^{\kern0.5pt q}\kern-0.5pt (\mathcal{O}_{\PP^n}^{\! r})$ is
  nonsingular.
\end{lemma}

\begin{proof} 
  Let $a_{n}$ denote the multiplicity of the part $n$ in the integer partition
  $\lambda$, and let $N_{\lambda}$ be the dimension of the lexicographic
  component of the Hilbert scheme
  $\Hilb^{\kern0.5pt p_\lambda} \kern-0.5pt (\PP^n)$. A general point on the
  lexicographic component (including the lexicographic point itself) of this
  Hilbert scheme corresponds to a residual flag; see \cite{RS97}*{Theorem~4.1}
  or \cite{SS23}*{\S1}. In particular, there exists an
  $N_{\lambda}$-dimensional family of ideals defining a full-dimensional
  neighbourhood of the lexicographic point such that each member has the form
  $I \coloneqq \langle f_0 \, g_1, f_0 \, g_2, \dotsc, f_0 \, g_m \rangle$,
  where the homogeneous polynomial $f_0 \in S$ has degree $a_{n}$.

  For the scheme
  $\Quot^{\kern0.5pt q}\kern-0.5pt (\mathcal{O}_{\PP^n}^{\! r})$,
  Proposition~\ref{p:dim} shows that the dimension of the tangent space at the
  lexicographic point is $\smash{N_{\lambda} + (r-1) \binom{n +
      a_{n}}{n}}$. Consider a nonzero matrix
  $\begin{bmatrix} f_1 & f_2 & \!\dotsb\! & f_{r-1} \end{bmatrix}$ whose
  entries are homogeneous polynomials in $S$ of degree $a_{n}$. Since
  $\smash{\dim_{\kk} S_{a_{n}} = \binom{n+a_{n}}{n}}$, the affine space
  parametrizing all such matrices has dimension
  $\smash{(r-1) \binom{n+a_n}{n}}$. The cokernel of the homogenous $S$-linear
  map to $S^r$ defined by the matrix
  \[
    \begin{bmatrix}
      f_0 \, g_1 & f_0 \, g_2 & \!\dotsb\! & f_0 \, g_m \\
      f_1 \, g_1 & f_1 \, g_2 & \!\dotsb\! & f_1 \, g_m \\
      f_2 \, g_1 & f_2 \, g_2 & \!\dotsb\! & f_2 \, g_m \\[-2pt]
      \vdots  & \vdots  & \!\ddots\! & \vdots  \\[-2pt]
      f_{r-1} \, g_1 & f_{r-1} \, g_2 & \!\dotsb\! & f_{r-1} \, g_m \\
    \end{bmatrix}
  \]
  corresponds to a point in
  $\Quot^{\kern0.5pt q}\kern-0.5pt (\mathcal{O}_{\PP^n}^{\! r})$. By
  construction, we obtain a family passing through the lexicographic point
  that has dimension $N_{\lambda} + (r-1) \binom{n + a_{n}}{n}$, so the
  lexicographic point on
  $\Quot^{\kern0.5pt q}\kern-0.5pt (\mathcal{O}_{\PP^n}^{\! r})$ is
  nonsingular.
\end{proof}

To treat arbitrary Quot schemes, we allow the rank of the ambient vector
bundle to increase.

\begin{lemma}
    \label{l:sec} 
    Let $k$ be a nonnegative integer. When the point in
    $\Quot^{\kern0.5pt q} (\mathcal{O}_{\PP^n}^{\! r})$ corresponding to the
    surjection $\mathcal{O}_{\PP^n}^{\! r} \to \mathcal{F}$ is nonsingular,
    the point in $\Quot^{\kern0.5pt q} (\mathcal{O}_{\PP^n}^{\! r+k})$
    corresponding to the surjection
    $\mathcal{O}_{\PP^n}^{\! r+k} \to \mathcal{F}$, obtained by precomposing
    with the projection
    $\mathcal{O}_{\PP^n}^{\! r+k} \to \mathcal{O}_{\PP^n}^{\! r}$, is also
    nonsingular.
\end{lemma}

\begin{proof} 
  Let $\mathcal{K}$ be the kernel of the surjection
  $\mathcal{O}^{\! r}_{\PP^n} \to \mathcal{F}$. The dimension of the tangent
  space to $\Quot^{\kern0.5pt q} (\mathcal{O}_{\PP^n}^{\! r})$ at the point
  corresponding to $\mathcal{O}_{\PP^n}^{\! r} \to \mathcal{F}$ is
  $h^0 \kern-0.5pt \bigl( \PP^n, \sHom(\mathcal{K}, \mathcal{F}) \kern-1.0pt
  \bigr)$. Similarly, the dimension of the tangent space to
  $\Quot^{\kern0.5pt q} (\mathcal{O}_{\PP^n}^{\! r+k})$ at the point
  corresponding to $\mathcal{O}_{\PP^n}^{\! r+k} \to \mathcal{F}$ is
  \[
    h^0 \kern-0.5pt \bigl( \PP^n, 
    \sHom ( \mathcal{K} \oplus \mathcal{O}^{\! k}_{\PP^{n}}, \mathcal{F}) 
    \kern-1.0pt \bigr) 
    = h^{0} \kern-0.5pt \bigl( \PP^n, \sHom ( \mathcal{K}, \mathcal{F})
    \kern-1.0pt \bigr) + k \, h^0 (\PP^n, \mathcal{F}) \, ,
  \]
  because $\sHom( \mathcal{O}^{\! 1}_{\PP^n}, -)$ is the identity functor.

  The hypothesis that the given point in
  $\Quot^{\kern0.5pt q} (\mathcal{O}_{\PP^n}^{\! r})$ is nonsingular means
  that there exists a family
  $\Phi \colon \mathcal{O}_{\PP^n}^{\! r} \to \mathcal{F}$ of surjections
  having dimension
  $h^0 \kern-0.5pt \bigl( \PP^n, \sHom(\mathcal{K}, \mathcal{F}) \kern-1.0pt
  \bigr)$ and passing through this point. Since the maps
  $\mathcal{O}_{\PP^n}^{\! 1} \to \mathcal{F}$ are parametrized by the global
  sections of $\mathcal{F}$, there also exists a family
  $\Psi \colon \mathcal{O}_{\PP^n}^{\! k} \to \mathcal{F}$ having dimension
  $k \, h^0 (\PP^n, \mathcal{F})$ and containing the zero map.  Hence, the
  copairing map
  $\begin{bmatrix} \Phi & \Psi \end{bmatrix} \colon \mathcal{O}_{\PP^n}^{\! r}
  \oplus \mathcal{O}_{\PP^n}^{\! k} \to \mathcal{F}$ (also known as the
  induced map arising from the universal property of the coproduct) determines
  a family that has dimension
  $h^{0} \kern-0.5pt \bigl( \PP^n, \sHom ( \mathcal{K}, \mathcal{F})
  \kern-1.0pt \bigr) + k \, h^0 (\PP^n, \mathcal{F})$ and passes through the
  given point in $\Quot^{\kern0.5pt q} (\mathcal{O}_{\PP^n}^{\! r+k})$.  We
  conclude that the desired point in
  $\Quot^{\kern0.5pt q} (\mathcal{O}_{\PP^n}^{\! r+k})$ is also smooth.
\end{proof}

Generalizing \cite{RS97}*{Theorem~1.4}, we exhibit a smooth point of any
$\Quot^{\kern0.5pt q} \kern-0.5pt (\mathcal{O}_{\PP^n}^{\! r})$.

\begin{theorem}
  \label{t:smooth} 
  For any nonempty Quot scheme
  $\Quot^{\kern0.5pt q}(\mathcal{O}^{\! r}_{\PP^n})$, the lexicographic point
  is smooth.
\end{theorem}

\begin{proof} 
  Proposition~\ref{p:nonempty} shows that the scheme
  $\Quot^{\kern0.5pt q}(\mathcal{O}^{\! r}_{\PP^n})$ is nonempty if and only
  if there exists a nonnegative integer $s$ and an integer partition
  $\lambda \coloneqq (\lambda_{1}, \lambda_{2}, \dotsc, \lambda_{e})$ such
  that $r \geqslant s$, $n \geqslant \lambda_{1}$, and
  $q(t) = s \binom{t+n}{n} + p_{\lambda}(t) \in \QQ[t]$. When $r = s$, we must
  have $\lambda = ()$ and the Quot scheme is a simple point and, thereby,
  smooth. When $r = s + 1$, Lemma~\ref{l:s+1} demonstrates that the
  lexicographic point on the scheme
  $\Quot^{\kern0.5pt q}(\mathcal{O}^{\! r}_{\PP^n})$ is smooth. Furthermore,
  for any nonnegative integer $k$, Lemma~\ref{l:sec} establishes that the
  lexicographic point on the scheme
  $\Quot^{\kern0.5pt q}(\mathcal{O}^{\! r+k}_{\PP^n})$ is also smooth.
\end{proof}

As an immediate application, we generalize \cite{Sta20}*{Lemma~5.6}.

\begin{corollary}
  \label{c:easy}
  Let $s$ be a nonnegative integer, let
  $\lambda \coloneqq (\lambda_1, \lambda_2, \dotsc, \lambda_{e})$ be an
  integer partition such that $r > s$, $n \geqslant \lambda_{1}$, and consider
  the polynomial
  $q(t) \coloneqq s \binom{t+n}{n} + p_{\lambda}(t) \in \QQ[t]$. Assuming that
  $\lambda_e \geqslant 2$, $\lambda = (1)$, or $\lambda = ()$, the scheme
  $\Quot^{\kern0.5pt q}\kern-0.5pt (\mathcal{O}_{\PP^n}^{\! r})$ is smooth and
  irreducible.
\end{corollary}

\begin{proof}
  When $\lambda_e \geqslant 2$, $\lambda = (1)$, or $\lambda = ()$,
  Corollary~\ref{c:!b} establishes that
  $\Quot^{\kern0.5pt q}\kern-0.5pt (\mathcal{O}_{\PP^n}^{\! r})$ has a unique
  Borel-fixed point, namely the lexicographic point. Since
  Theorem~\ref{t:smooth} shows that it is nonsingular, the lexicographic point
  on this Quot scheme cannot lie on an intersection of irreducible
  components. Therefore, the scheme
  $\Quot^{\kern0.5pt q}\kern-0.5pt (\mathcal{O}_{\PP^n}^{\! r})$ has a unique,
  generically nonsingular, irreducible component. Furthermore, the generic
  submodule of initial terms described in \cite{Par96}*{Proposition~4} yields
  a one-parameter family on
  $\Quot^{\kern0.5pt q}\kern-0.5pt (\mathcal{O}_{\PP^n}^{\! r})$ connecting
  any point to the unique Borel-fixed point. Since the dimension of the
  tangent space at a point in family is an upper-semicontinuous function, we
  conclude that every point on the Quot scheme is nonsingular.
\end{proof}

\section{A Family On The Lexicographic Component}
\label{s:fam}

\noindent
As a complement to the previous section, we exhibit an explicit
full-dimensional family on the lexicographic component of
$\Quot^{\kern0.5pt p_{\lambda}} \kern-0.5pt (\mathcal{O}_{\PP^n}^{2})$ where
the integer partition is $\lambda = (n^{e-1},1)$ with $e \geqslant 2$. This
covers the second and third conditions in the ($r \geqslant s+2)$-case of
Theorem~\ref{t:main} when $r = 2$ and $s = 0$.  Notably, the proof of
Theorem~\ref{t:smooth} avoids many of the intricacies arising in these
examples.  Throughout this section, the base scheme is $\Spec(\kk)$ and
$S \coloneqq \kk[x_0, x_1, \dotsc, x_n]$ is still the homogeneous coordinate
ring of $\PP^n$.

Consider the nonempty subpartition $\mu \coloneqq (n^{e-1})$ of the integer
partition $\lambda$.  Let
\[
  f \coloneqq x_{0}^{e-1} + \sum_{\bm{u}} z_{\bm{u}} \, \bm{x}^{\bm{u}}
\]
be a generic homogeneous polynomial in $S$ where the sum runs over all
monomials in $S$ of degree $e-1$ except $x_0^{e-1}$.  Setting
$m \coloneqq \binom{n+e-1}{e-1} - 1$, the $m$-tuple
$( \dotsc, z_{\bm{u}}, \dotsc)$ of indeterminate coefficients is a point in
$\AA^{\!m}$.  Similarly, for each index $i$ satisfying
$0 \leqslant i \leqslant n-1$, let $g_{i} \coloneqq x_i - y_i \, x_n$ be a
linear polynomial in $S$.  We also regard the $n$-tuple
$(y_0, y_1, \dotsc, y_{n-1})$ of coefficients as a point in $\AA^{\!n}$.

Given this notation, the next proposition describes the desired family
explicitly.

\begin{proposition}
  \label{p:fam} 
  Consider the integer partition $\lambda \coloneqq (n^{e-1}, 1)$ where
  $e \geqslant 2$.  Let $[\alpha \mathbin{:} \beta]$ and
  $[\gamma \mathbin{:} \delta]$ be points in $\PP^1$, and let
  $U \subset \PP^1 \times \PP^1$ be the open complement of the subvariety
  defined by the equation $\alpha \delta - \beta \gamma = 0$. For any point in
  $\mathbb{A}^{\!m} \times \mathbb{A}^{\!n} \times U$, the cokernel of the
  homogeneous $S$-linear map
  \begin{align*}
    S^2
    &\xleftarrow{{ \quad
      \mathbf{F} \coloneqq
      \begin{bmatrix}
        \alpha \, g_0 & \alpha \, g_1 & \dotsb & \alpha \, g_{n-1} & \gamma \, f  \\
        \beta \, g_0  & \beta \, g_1  & \dotsb & \beta \, g_{n-1}  & \delta \, f \\
      \end{bmatrix} \quad}}
    S^n(-1) \oplus S^1(-e+1) 
  \end{align*}
  defines a point in
  $\Quot^{\kern0.5pt p_{\lambda}} \kern-0.5pt (\mathcal{O}_{\PP^n}^{\! 2})$,
  and this family identifies the product $\AA^{\!m} \times \AA^{\!n} \times U$
  with an open dense subset of the lexicographic component.  Moreover, this
  family contains the Borel-fixed point corresponding to the submodule
  $L(1) \oplus L(\mu)$ in $S^2$, and the lexicographic point lies in its
  closure.
\end{proposition}

\begin{proof} 
  For each point in $\AA^{\!m} \times \AA^{\!n}$, consider the cokernel of the
  homogeneous $S$-linear map
  \begin{align*}
    S^2
    &\xleftarrow{{ \quad
      \begin{bmatrix}
        g_0 & g_1 & \dotsb & g_{n-1} & 0  \\
        0   & 0   & \dotsb & 0 & f \\
      \end{bmatrix} \quad}}
    S^n(-1) \oplus S^1(-e+1) \, . 
  \end{align*}   
  By construction, this quotient of $S^2$ is a direct sum of the comodules for
  the saturated ideals corresponding to the point
  $[y_0 \mathbin{:} y_1 \mathbin{:} \dotsb \mathbin{:} y_{n-1} \mathbin{:} 1]$
  in $\PP^n$ and the hypersurface in $\PP^n$ defined by the homogeneous
  polynomial $f$ in $S$.  It follows that each of these quotient modules
  corresponds to a distinct quotient sheaf of $\mathcal{O}_{\PP^n}^{\! 2}$,
  and the Hilbert polynomial of the fibre over a point in
  $\AA^{\!m} \times \AA^{\!n}$ is $p_{\lambda}$ where
  $\lambda = (n^{e-1}, 1)$; see \cite{SS23}*{Remark~1.9}.  Multiplying on the
  left by the invertible matrix
  \[
    \begin{bmatrix} 
      \alpha & \gamma \\ 
      \beta  & \delta \\
    \end{bmatrix} 
    \quad \text{where $([ \alpha \mathbin{:} \beta], [\gamma \mathbin{:} \delta]) \in U$,}
  \]
  we obtain the $S$-linear map $\mathbf{F}$. Together \cite{SS23}*{Remark~2.7}
  and Proposition~\ref{p:dim} demonstrate that
  $\dim (\AA^{\!m} \times \AA^{\!n} \times U) = m+n+2 =
  \smash{\binom{n+e-1}{e-1}} + n + 1 = N_{\lambda} + 2$, which equals the
  dimension of the tangent space to
  $\Quot^{\kern0.5pt p_{\lambda}} \kern-0.5pt (\mathcal{O}_{\PP^n}^{\! 2})$ at
  the lexicographic point. Hence, the product
  $\AA^{\!m} \times \AA^{\!n} \times U$ is identified with an open subset in
  the Quot scheme
  $\Quot^{\kern0.5pt p_{\lambda}} \kern-0.5pt (\mathcal{O}_{\PP^n}^{\! 2})$.
    
  To finish, we relate this family to Borel-fixed points on
  $\Quot^{\kern0.5pt p_{\lambda}} \kern-0.5pt (\mathcal{O}_{\PP^n}^{2})$. This
  scheme has at least two Borel-fixed points corresponding to the submodules
  $S^1 \oplus L(\lambda) \subseteq S^2$ and $L(1) \oplus L(\mu) \subseteq S^2$
  where $L(1) = \langle x_0, x_1, \dotsc, x_{n-1} \rangle$ and
  $L(\mu) = \langle x_{0}^{e-1} \rangle$; also see Lemma~\ref{l:2b}. When
  $( \dotsc, z_{\bm{u}}, \dotsc)= \bm{0} \in \AA^{\!m}$,
  $(y_0, y_1, \dotsc, y_{n-1}) = \bm{0} \in \AA^{\!n}$,
  $[\alpha \mathbin{:} \beta] = [1 \mathbin{:} 0] \in \PP^1$, and
  $[\gamma \mathbin{:} \delta] = [0 \mathbin{:} 1] \in \PP^1$, we see
  straightaway that submodule $L(1) \oplus L(\mu) \subseteq S^2$ belongs to
  this family. On the other hand, let
  $( \dotsc, z_{\bm{u}}, \dotsc)= \bm{0} \in \AA^{\!m}$,
  $(y_0, y_1, \dotsc, y_{n-1}) = (y_0,0,0 \dotsc, 0) \in \AA^{\!n}$ where
  $y_0 \neq 0$, $[\alpha \mathbin{:} \beta] = [1 \mathbin{:} \beta] \in \PP^1$
  where $\beta \neq 0$, and
  $[\gamma \mathbin{:} \delta] = [1 \mathbin{:} \beta (y_{0}^{e} +1)] \in
  \PP^1$.  In this situation, one verifies that the columns in the matrix
  \[
    \setcounter{MaxMatrixCols}{11}
    \begin{bsmallmatrix}
      0 & 0 & \dotsb & 0 & 0 & x_n^{e-1} 
      & x_{n-1}^{} & x_{n-2}^{} & \dotsb & x_1^{} & x_0^{} - y_0^{} \, x_n^{}  \\
      x_0^{e-1} x_{n-1}^{} & x_0^{e-1} x_{n-2}^{} & \dotsb & x_0^{e-1} x_1^{} 
      & x_0^{e} \!-\! y_0^{} \, x_0^{e-1} x_{n\vphantom{1}}^{} 
      & \beta y_{0}^{} \, x_0^{e-1} \!-\! \beta \, x_{n\vphantom{1}}^{e-1}
      &  \beta \, x_{n-1}^{} & \beta \, x_{n-2}^{} & \dotsb & \beta \,
      x_{1}^{}
      & \beta (x_0^{} \!-\! y_0^{} \, x_n^{}) \\
    \end{bsmallmatrix}
  \]
  form a Gröbner basis (where the monomial order has position over term with
  position down and the graded reverse lexicographic order on terms) for the
  image of $\mathbf{F}$.  Remark~\ref{r:gen} establishes that
  $L(\lambda) = \langle x_{0}^{e}, x_{0}^{e-1} x_{1}^{}, x_{0}^{e-1} x_{2}^{},
  \dotsc, x_{0}^{e-1} x_{n-1}^{} \rangle$.  Letting $y_0 \to 0$,
  $\beta \to 0$, and saturating with respect to the irrelevant ideal
  $\langle x_0, x_1, \dotsc, x_n \rangle$, we see that the submodule
  $S^1 \oplus L(\lambda) \subseteq S^2$ is a flat limit of this family.  Thus,
  the lexicographic point on
  $\Quot^{\kern0.5pt p_{\lambda}} \kern-0.5pt (\mathcal{O}_{\PP^n}^{2})$ lies
  in the closure of the family.  We conclude that $\mathbf{F}$ defines a
  full-dimensional family on the lexicographic component of
  $\Quot^{\kern0.5pt p_{\lambda}} \kern-0.5pt (\mathcal{O}_{\PP^n}^{2})$.
\end{proof}

\section{Some Recognizable Smooth Quot Schemes}
\label{s:Gr}

\noindent
Unlike Hilbert schemes, surprisingly few smooth
$\Quot^{\kern0.5pt q} \kern-0.5pt (\mathcal{O}_{\PP^n}^{\! r})$ can already be
found in the literature; compare with \cite{SS23}*{p.~284}.  To compensate for
this shortcoming, this section showcases three countable collections of
nonsingular Quot schemes by identifying them with other parameter
spaces. Throughout this section, fix a base scheme $Y$ and let $X$ denote a
scheme over $Y$. Given integers $r$ and $s$, the \emph{Grassmannian}
$\Gr(s, r)$ is the projective scheme over $Y$ whose $X$-valued points are
isomorphism classes of locally free $\mathcal{O}_{X}$-module quotients of
$\mathcal{O}_{X}^{\! r}$ having rank $s$; see \cite{GD71}*{\S\S9.7}.

When the integer partition is empty, the Quot schemes are well-known smooth
varieties.

\begin{proposition}
  \label{p:gr} 
  For all integers $r$ and $s$ satisfying $0 \leqslant s \leqslant r$, we have
  $\smash{\Quot^{\kern0.5pt s \binom{t+n}{n}} \kern-0.5pt
    (\mathcal{O}_{\PP^n}^{\!r})} \cong \Gr(s, r)$.
\end{proposition}

\begin{proof} 
  Set $q(t) \coloneqq \binom{t+n}{n} \in \QQ[t]$. There exists a morphism
  $\psi \colon \Gr(s,r) \to \Quot^{\kern0.5pt s \, q} \kern-0.5pt
  (\mathcal{O}_{\PP^n}^{\!r})$ defined by sending an $X$\nobreakdash-valued
  point $\mathcal{O}_{X}^{\! r} \to \mathcal{G}$ to the pullback
  $\smash{\mathcal{O}_{\PP^{n}_{X}}^{\! r}} \cong
  \pi_{X}^{*}(\mathcal{O}_{X}^{\! r}) \to \pi_{X}^*(\mathcal{G})$, where
  $\pi_X \colon \mathbb{P}^n_X \to X$ is the structure map. Conversely, an
  $X$-valued point of
  $\Quot^{\kern0.5pt s \, q} \kern-0.5pt (\mathcal{O}_{\PP^n}^{\!r})$ is a
  quotient $\smash{\mathcal{O}_{\PP^{n}_{X}}^{\!r}} \to \mathcal{F}$, where
  $\mathcal{F}$ is flat over $X$ and the Hilbert polynomial on each fibre
  equals $s \, q(t)$.  For any sufficiently large integer $m$, the relative
  version of Serre's Vanishing Theorem shows that
  $R^i \kern-0.5pt ( \kern-0.5pt \pi_{X} \kern-0.75pt )^{}_{*} \kern-0.5pt
  \bigl( \mathcal{F}(m) \kern-1.0pt \bigr) = 0$ for all $i > 0$.  Hence, the
  base change theorem implies that
  $( \kern-0.5pt \pi_{X} \kern-0.5pt )^{}_{*} \kern-0.5pt \bigl(
  \mathcal{F}(m) \kern-1.0pt \bigr)$ is locally free of rank $s \, q(m)$; see
  \cite{Mum08}*{Corollary~II.5.2}.  As
  $\PP^n_{X} = \operatorname{Proj}_{X}(\mathcal{O}_{X}^{n+1})$, the projection
  formula establishes that
  \[
    ( \kern-0.5pt \pi_{X} \kern-0.5pt )^{}_{*} \kern-0.5pt \bigl( \mathcal{F}(m) \kern-1.0pt \bigr) 
    \cong ( \kern-0.5pt \pi_{X} \kern-0.5pt )^{}_{*} (\mathcal{F}) \otimes \operatorname{Sym}^{m}(\mathcal{O}_{X}^{n+1}) 
    \cong ( \kern-0.5pt \pi_{X} \kern-0.5pt )^{}_{*} (\mathcal{F}) \otimes \mathcal{O}_{X}^{q(m)} \, . 
  \]
  Since the direct sum $\smash{\mathcal{O}_{X}^{q(m)}}$ is a faithfully flat
  sheaf, we deduce that the direct image
  $( \kern-0.5pt \pi_X \kern-0.5pt )^{}_{*}(\mathcal{F})$ is a locally-free
  quotient of $\mathcal{O}_{X}^{\! r}$ of rank $s$.  We conclude that every
  $X$-valued point of
  $\Quot^{\kern0.5pt s \, q} \kern-0.5pt (\mathcal{O}_{\PP^n}^{\!r})$ is the
  pullback of an $X$-point in $\Gr(s,r)$, so $\psi$ is an isomorphism.
\end{proof}

Expanding on this idea produces closed immersions that relate Grassmannians,
Hilbert schemes, and Quot schemes; compare with
\cite{Ser06}*{Proposition~4.6.2}.

\begin{lemma} 
  \label{l:morp}
  Let $k$ and $r$ be integers satisfying $1 \leqslant k \leqslant r$. For any
  nonzero polynomial $q \in \QQ[t]$, there exists a closed immersion
  \[
    \varphi \colon \Gr(k, r) \mathbin{\times} \Hilb^{\kern0.5pt q} \kern-0.5pt (\PP^n) 
    \to \Quot^{\kern0.5pt k \, q} \kern-0.5pt (\mathcal{O}_{\PP^{n}}^{\! r})
  \]  
  that sends an $X$-valued point
  $(\mathcal{O}_{X}^{\! r} \to \mathcal{G} \kern-1.0pt, \, \kern0.5pt Z
  \hookrightarrow \PP^{n}_{X} )$ in the source to the composition of the
  morphisms
  $\smash{\mathcal{O}_{\PP^{n}_{X}}^{\! r}} \cong
  \pi_{X}^{*}(\mathcal{O}_{X}^{\! r}) \to \pi_{X}^*(\mathcal{G})$ and
  $\pi_{X}^*(\mathcal{G}) \to \pi_{X}^{*}(\mathcal{G})
  \mathbin{\smash{\otimes_{\mathcal{O}_{\PP^n_{X}}}}} \kern-2.0pt
  \mathcal{O}_{Z}^{}$, where $\mathcal{O}_{Z}^{}$ denotes the structure sheaf
  on the closed subscheme $Z$ in $\PP^n_{X}$ and
  $\pi_{X} \colon \PP^{n}_{X} \to X$ is the structure map.
\end{lemma}

\begin{proof} 
  Consider a closed subscheme $Z$ in $\PP^n_{X}$ with Hilbert polynomial
  $q$. For any locally-free quotient $\mathcal{O}_{X}^{\! r} \to \mathcal{G}$
  having rank $k$, the pullback $\pi_X^{*}(\mathcal{G})$ is a locally-free
  quotient of
  $\smash{\mathcal{O}_{\PP^{n}_{X}}^{\! r}} \cong
  \pi_{X}^{*}(\mathcal{O}_{X}^{\! r})$ also having rank $k$. Tensoring the
  canonical surjection $\mathcal{O}_{\PP^n} \to \mathcal{O}_{Z}$ with
  $\pi_X^{*}(\mathcal{G})$ gives the surjection
  $\pi_{X}^*(\mathcal{G}) \to \pi_{X}^{*}(\mathcal{G}) \otimes
  \mathcal{O}_{Z}$.  Since $\pi_X^{*}(\mathcal{G})$ is locally free of rank
  $k$, the Hilbert polynomial of the coherent sheaf
  $\pi_{X}^{*}(\mathcal{G}) \otimes \mathcal{O}_{Z}$ equals $k \, q$.  Hence,
  this composition operation defines a map into
  $\Quot^{\kern0.5pt k \, q} \kern-0.5pt (\mathcal{O}_{\PP^{n}}^{\! r})$.

  We claim that the morphism $\varphi$ is injective on all $X$-valued
  points. By Proposition~\ref{p:gr}, we may replace $\Gr(k,r)$ with
  $\smash{\Quot^{\kern0.5pt k \binom{t+n}{n}} \kern-0.5pt
    (\mathcal{O}_{\PP^n}^{\!r})}$. Consider an $X$-valued point
  $\smash{\mathcal{O}_{\PP^{n}_{X}}^{\! r}} \to \mathcal{F}$ of this Quot
  scheme and an $X$-valued point
  $\mathcal{O}_{\PP^{n}_{X}} \to \mathcal{O}_{Z}$ of the Hilbert scheme
  $\Hilb^{\kern0.5pt q} \kern-0.5pt (\PP^n)$. The morphism $\varphi$ sends
  this pair to the quotient
  $\smash{\mathcal{O}_{\PP^{n}_{X}}^{\! r}} \to \mathcal{F} \otimes
  \mathcal{O}_{Z}$. Since $\mathcal{F}$ is locally-free on $\PP^{n}_{X}$,
  there exists an open cover $\{ U_{j} \}_{j \in J}$ of $\PP^{n}_{X}$ that
  trivializes $\mathcal{F}$.  Over each such open immersion
  $\iota_{\kern-0.5pt j} \colon U_{j} \to \PP^{n}_{X}$, we have
  \[
    \iota^{\!*}_{\kern-0.5pt j} ( \mathcal{F} \otimes \mathcal{O}_{Z})
    = \bigoplus_{i=1}^{k} \iota^{\!*}_{\kern-0.5pt j}(\mathcal{O}_{Z})
    = \bigoplus_{i=1}^{k} \mathcal{O}_{Z \cap U_{j}} \, . 
  \]
  Given pairs $(\mathcal{F}, \mathcal{O}_{Z})$ and
  $(\mathcal{F}', \mathcal{O}_{Z'})$ of $X$-valued points that have the same
  image under $\varphi$, the pullback to a common trivialization for
  $\mathcal{F}$ and $\mathcal{F}'$ shows that the closed subschemes $Z$ and
  $Z'$ are equal on an open covering of $\PP^{n}_{X}$. It follows that
  $Z = Z'$.
    
  To see that $\mathcal{F} = \mathcal{F}'$, consider an arbitrary affine open
  subset $V \coloneqq \Spec(R) \subseteq X$. Working over the homogeneous
  coordinate ring $S \coloneqq R[x_0, x_1, \dotsc, x_n]$ of $\PP^{n}_{V}$,
  there is a commutative diagram
  \[
    \begin{tikzcd}
      \bigoplus _{i=1}^r S \ar[r, "\sigma"] \ar[d, "\sigma'" left] 
      & F \ar[d, "\tau"] \\
      F' \ar[r, "\tau'"] 
      & F \otimes_S S/I \, ,
    \end{tikzcd}
  \]
  of $\ZZ$-graded $S$-modules, where the maps $\sigma$ and $\sigma'$ are
  surjective, the sheaves associated to $F$ and $F'$ are the restrictions of
  $\mathcal{F}$ and $\mathcal{F'}$ to $\PP^{n}_{V}$, and the homogeneous ideal
  $I$ in $S$ corresponds to the restriction of $Z$ to $\PP^{n}_{V}$. The
  assumption on the pairs $(\mathcal{F}, \mathcal{O}_{Z})$ and
  $(\mathcal{F}', \mathcal{O}_{Z'})$ ensures that
  $F \otimes_S S/I \cong F' \otimes_S S/I$. Because $Z$ is flat and nonempty
  over $X$, neither $\tau$ nor $\tau'$ have any nonzero elements of degree
  zero in their kernels. Hence, the equality
  $\tau \circ \sigma = \tau' \circ \sigma'$ implies that the degree-zero
  elements in the kernels of $\sigma$ and $\sigma'$ are equal. We deduce that
  $\mathcal{F}$ and $\mathcal{F}'$ are the same quotient sheaf of
  $\mathcal{O}^{\! r}_{\PP^{n}_{X}}$.

  The previous two paragraphs demonstrate that $\varphi$ is a
  monomorphism. Since the schemes $\Gr(k,r)$,
  $\Hilb^{\kern0.5pt q} \kern-0.5pt (\PP^n)$, and
  $\Quot^{\kern0.5pt k \, q} \kern-0.5pt (\mathcal{O}_{\PP^{n}}^{\! r})$ are
  all projective, the map $\varphi$ is proper. It follows that $\varphi$ is a
  closed immersion; see \cite{Gro66}*{Proposition~8.11.5} or
  \cite{SP24}*{\href{https://stacks.math.columbia.edu/tag/04XV}{Tag~04XV}}.
\end{proof}

\begin{example} 
  Suppose that $k = 1$ and $q(t) = 1 \in \QQ[t]$. A single point is the only
  closed subscheme of $\PP^{n}$ with Hilbert polynomial $1$, so
  $\Hilb^{\kern0.5pt 1} \kern-0.5pt (\PP^n) = \PP^n$. As
  $\Gr(1, r) = \PP^{r-1}$, Example~\ref{e:1pt} shows that the closed immersion
  $\varphi \colon \PP^{r-1} \mathbin{\times} \PP^{n} \to \Quot^{\kern0.5pt 1}
  \kern-0.5pt (\mathcal{O}_{\PP^{n}}^{\! r})$ is an isomorphism.
\end{example}

\begin{example}
  When $q(t) \coloneqq \binom{t+n}{n} \in \QQ[t]$, the scheme
  $\operatorname{Hilb}^q(\PP^n)$ is just a single point. Hence,
  Proposition~\ref{p:gr} shows that the closed immersion
  $\varphi \colon \!\Gr(k,r) \to \Quot^{\kern0.5pt k \, q} \kern-0.5pt
  (\mathcal{O}_{\PP^{n}}^{\! r})$ is the isomorphism.
\end{example}

\begin{remark} 
  One may further generalize Lemma~\ref{l:morp} by replacing the Hilbert
  scheme with a Quot scheme.  For any coherent sheaf $\mathcal{E}$ on
  $\PP^{n}$, essentially the same operation defines a morphism
  $\varphi \colon \Gr(k, r) \mathbin{\times} \Quot^{\kern0.5pt q} \kern-0.5pt
  (\mathcal{E}) \to \Quot^{\kern0.5pt k \, q} \kern-0.5pt (
  \mathcal{O}_{\PP^n}^{\! r} \otimes \mathcal{E})$.
\end{remark}

As a counterpart to $\lambda = ()$, we recognize many Quot schemes as products
when $s = 0$. More precisely, we claim that an absence of parts equal to $1$
in the integer partition $\lambda$ guarantees that the scheme
$\Quot^{p_{\lambda}}(\mathcal{O}_{\PP^n}^{\! r})$ is a trivial projective
bundle over the Hilbert scheme
$\Hilb^{\kern0.5pt p_{\lambda}} \kern-0.5pt (\PP^n)$. Equivalently, the
following theorem shows that the map
$\varphi \colon \PP^{r-1} \times \Hilb^{\kern0.5pt p_{\lambda}} \kern-0.5pt
(\PP^n) \to \Quot^{\kern0.5pt p_{\lambda}} \kern-0.5pt
(\mathcal{O}_{\PP^n}^{\! r})$, arising from Lemma~\ref{l:morp} when $k = 1$,
is an isomorphism under this hypothesis.

\begin{theorem}
  \label{t:prod}
  For any integer partition
  $\lambda \coloneqq (\lambda_1, \lambda_2, \dotsc, \lambda_{e})$ having no
  parts equal to $1$, we have
  \[ 
    \PP^{r-1} \times \Hilb^{\kern0.5pt p_{\lambda}} \kern-0.5pt (\PP^n) 
    \cong \Quot^{\kern0.5pt p_{\lambda}} \kern-0.5pt (\mathcal{O}_{\PP^n}^{\! r}) \, .
  \]
\end{theorem}

\begin{proof}
  Corollary~\ref{c:easy} shows that
  $\Quot^{\kern0.5pt p_\lambda} \kern-0.5pt (\mathcal{O}_{\PP^n}^{\! r})$ is
  smooth and irreducible. Moreover, Corollary~\ref{p:dim} implies that this
  Quot scheme has dimension $N_{\lambda} + r-1$ where
  $N_{\lambda} \coloneqq \dim \, \Hilb^{\kern0.5pt p_{\lambda}} \kern-0.5pt
  (\PP^n)$.  Since $\Gr(1,r) = \PP^{r-1}$, Lemma~\ref{l:morp} with $k=1$ gives
  the closed immersion
  \[ 
    \varphi \colon \PP^{r-1} \times \Hilb^{\kern0.5pt p_{\lambda}} \kern-0.5pt (\PP^n)
    \to \Quot^{\kern0.5pt p_{\lambda}} \kern-0.5pt (\mathcal{O}_{\PP^n}^{\! r}) \, ,
  \]
  which must be an isomorphism.
\end{proof}

The third collection prescribes both the nonnegative integer $s$ and the
integer partition $\lambda$.

\begin{proposition} 
  \label{p:3rd}
  When $r = s+1$ and $\lambda \coloneqq (n^e)$, we have 
  $\Quot^{\kern0.5pt q} \kern-0.5pt (\mathcal{O}_{\PP^n}^{\! r}) 
  \cong \mathbb{P}^m$ where $m \coloneqq r \binom{n+e}{n} - 1$.
\end{proposition}

\begin{proof} 
  Consider a nonzero $(1 \times r)$-matrix $\mathbf{G}$ whose entries are
  homogeneous polynomials in $S$ of degree $e$. The cokernel of the
  homogeneous $S$-linear map from $S(-e)$ to $S^r$ defined by $\mathbf{G}$
  defines a point in
  $\Quot^{\kern0.5pt q} \kern-0.5pt (\mathcal{O}_{\PP^n}^{\! r})$, because
  \[
    q(t) 
    = (r-1) \binom{t+n}{n} + p_{\lambda}(t)
    = (r-1) \binom{t+n}{n} + \sum_{i=1}^{e} \binom{t+n-i}{n-1}
    = r \binom{t+n}{n} - \binom{t+n-e}{n} \, . 
  \]
  Moreover, multiplying $\mathbf{G}$ by a nonzero scalar does not change
  either its image in $S^r$ or the correspond point in the Quot scheme. Since
  $\dim_{\kk} S_e = \binom{n+e}{n}$ and $m \coloneqq r \binom{n+e}{n} - 1$,
  the family arising from these matrices identifies $\PP^m$ with a subset of
  $\Quot^{\kern0.5pt q} \kern-0.5pt (\mathcal{O}_{\PP^n}^{\! r})$.

  When $\lambda = (n^{e})$, the Hilbert scheme
  $\Hilb^{\kern0.5pt p_{\lambda}} \kern-0.5pt (\PP^n)$ parameterizes
  hypersurfaces of degree $e$ in $\PP^n$. It follows that
  $N_{\lambda} = \binom{n+e}{n} - 1$; also see \cite{SS23}*{Remark~2.7}. As
  $\lambda_{e} \geqslant 2$, Corollary~\ref{c:easy} shows that the scheme
  $\Quot^{\kern0.5pt q} \kern-0.5pt (\mathcal{O}_{\PP^n}^{\! r})$ is smooth
  and irreducible. Since $r = s+1$, Proposition~\ref{p:dim} establishes that
  this Quot scheme has dimension $m = N_{\lambda} + (r-1) \binom{n+e}{n}$. We
  conclude that
  $\Quot^{\kern0.5pt q} \kern-0.5pt (\mathcal{O}_{\PP^n}^{\! r}) \cong
  \mathbb{P}^m$.
\end{proof}

\begin{example}
  \label{e:rig}
  Together Example~\ref{e:TP} and Proposition~\ref{p:3rd} show that the Quot
  scheme that has a point corresponding to the twisted tangent bundle
  $\mathcal{T}_{\PP^n}(-1)$ is isomorphic to $\PP^m$, where
  $m \coloneqq n^2+2n$. Since $\dim \operatorname{PGL}(n+1,\kk) = n^2+2n$, one
  can show that any flat deformation of the tangent bundle on $\PP^n$ is
  isomorphic to the original bundle; compare with \cite{Siu91}*{Main Theorem}.
\end{example}

\section{Singular Quot Schemes}
\label{s:sing}

\noindent
This section identifies sufficient conditions on the nonnegative integer $s$
and the integer partition
$\lambda \coloneqq (\lambda_1, \lambda_2, \dotsc, \lambda_e)$ that guarantee
the scheme $\Quot^{\kern0.5pt q} \kern-0.5pt (\mathcal{O}_{\PP^n}^{\! r})$ is
singular, where
\[
  q(t) \coloneqq s \binom{t+n}{n} + p_{\lambda}(t) \in \QQ[t] \, .
\]
The base scheme remains $\Spec(\kk)$, every integer partition $\lambda$
continues to satisfy $n \geqslant \lambda_1$, and the polynomial ring
$S \coloneqq \kk[x_0,x_1, \dotsc, x_n]$ is still the homogeneous coordinate
ring of $\PP^n$.

When $r \geqslant s+2$, Corollary~\ref{c:!b} shows that the conditions
$\lambda_{e} \geqslant 2$, $\lambda = (1)$, or $\lambda = ()$ are the same as
the scheme $\Quot^{\kern0.5pt q} \kern-0.5pt (\mathcal{O}_{\PP^n}^{\! r})$
having exactly one Borel-fixed point. Otherwise $\lambda_{e} = 1$ and
$e \geqslant 2$ guarantees that the scheme
$\Quot^{\kern0.5pt q} \kern-0.5pt (\mathcal{O}_{\PP^n}^{\! r})$ has at least
two Borel-fixed points. Specifically, the lexicographic point corresponds to
the submodule $S^{r-s-1} \oplus L(\lambda) \oplus 0^{s} \subseteq S^r$, and
another Borel-fixed point corresponds to the submodule
$S^{r-s-2} \oplus L(1) \oplus L \bigl( \kern-1.0pt (\lambda_1, \lambda_2,
\dotsc, \lambda_{e-1}) \kern-1.0pt \bigr) \oplus 0^{s} \subseteq S^{r}$, where
$L(1) \coloneqq L\bigl( \kern-1.0pt (1) \kern-1.0pt \bigr) = \langle x_0, x_1,
\dotsc, x_{n-1} \rangle$. Comparing the tangent spaces at these points is the
key to producing many singular Quot schemes.

\begin{lemma}
  \label{l:key} 
  Let $s$ be a nonnegative integer, let
  $\lambda \coloneqq (\lambda_1, \lambda_{2}, \dotsc, \lambda_{e-1}, 1)$ be an
  integer partition with $e \geqslant 2$ and at least one part equal to $1$,
  and consider the polynomial
  $q(t) \coloneqq s \binom{t+n}{n} + p_{\lambda}(t) \in \QQ[t]$. Assume that
  $r \geqslant s+2$.  Writing
  $\mu \coloneqq (\lambda_1, \lambda_{2}, \dotsc, \lambda_{e-1})$ for a
  subpartition of $\lambda$, the submodules
  $S^{r-s-1} \oplus L(\lambda) \oplus 0^{s} \subseteq S^{r}$ and
  $S^{r-s-2} \oplus L(1) \oplus L(\mu) \oplus 0^{s} \subseteq S^{r}$ determine
  two Borel-fixed points in the Quot scheme
  $\Quot^{\kern0.5pt q }\kern-0.5pt (\mathcal{O}_{\PP^n}^{\! r})$. The
  dimension of the tangent spaces to this scheme at these points are equal if
  and only if
  \[  
    h^0 \kern-1.0pt \Bigl( \kern-0.5pt \Hom_S \bigl(  L(\mu), \,
    S^1 \mathbin{\!\!/\!} L(1) \kern-1.0pt \bigr) \kern-3.0pt \Bigr) \kern-1.0pt
    + h^0 \kern-1.0pt \Bigl( \kern-0.5pt \Hom_S \bigl( L(1), \, 
    S^1 \mathbin{\!\!/\!} L(\mu) \kern-1.0pt \bigr) \kern-3.0pt \Bigr) \kern-1.0pt
    = h^0 \kern-0.5pt \bigl( S^1 \mathbin{\!\!/\!} L(\lambda) \kern-1.0pt \bigr) 
    \, .
  \]
  In particular, the scheme
  $\Quot^{\kern0.5pt q} \kern-0.5pt (\mathcal{O}_{\PP^n}^{\! r})$ is singular
  whenever this equality fails to hold.
\end{lemma}

\begin{proof} 
  The scheme $\Quot^{\kern0.5pt q} \kern-0.5pt (\mathcal{O}_{\PP^n}^{\! r})$
  is path-connected; see \cite{BdC88}*{Th\'eor\`eme~1} or
  \cite{Par96}*{Theorem~36}.  Having two points whose tangent spaces have
  different dimensions would, thereby, imply that this Quot scheme is
  singular, so the second part of the lemma follows from the first part.

  As in the proof of Proposition~\ref{p:dim}, the tangent space to this Quot
  scheme at the point corresponding to the submodule
  $M \coloneqq S^{r-s-2} \oplus L(1) \oplus L(\mu) \oplus 0^{s} \subseteq S^r$
  is the space of the global sections of the sheaf associated to the
  $\ZZ$\nobreakdash-graded $S$-module
  $\Hom_{S} \kern-0.5pt \bigl( M, S^1 \mathbin{\!\!/\!} L(1) \oplus S^1
  \mathbin{\!\!/\!} L(\mu) \oplus S^{s} \kern-1.0pt \bigr)$. Since the Hom
  functor commutes with finite direct sums, this $S$-module is isomorphic to
  \begin{align*}
    & S^{(r-s-2)s} \oplus \kern-1.0pt \left( 
      \bigoplus_{i=1}^{r-s-2} \kern-3.0pt  \left( 
      \frac{S^1}{L(1)} \oplus \frac{S^1}{L(\mu)}
      \kern-1.0pt \right) \kern-4.0pt \right) 
      \oplus \left( \bigoplus_{j=1}^{s} \Bigl(
      \Hom_{S} \kern-0.5pt \bigl( L(1), S \bigr) 
      \oplus \Hom_{S} \kern-0.5pt \bigl( L(\mu), S \bigr) 
      \kern-2.0pt \Bigr) \kern-4.0pt \right) \\[-5pt]
    &\phantom{ww} \oplus \Hom_{S} \kern-1.0pt \left( \kern-1.0pt
      L(1), \frac{S^1}{L(1)} \kern-1.0pt \right)
      \oplus \Hom_{S} \kern-1.0pt \left( \kern-0.5pt 
      L(1),  \frac{S^1}{L(\mu)} \kern-1.0pt \right) 
      \oplus \Hom_{S} \kern-1.0pt \left( \kern-1.0pt 
      L(\mu), \frac{S^1}{L(1)}  \kern-1.0pt \right)
      \oplus \Hom_{S} \kern-1.0pt \left( \kern-1.0pt 
      L(\mu), \frac{S^1}{L(\mu)} \kern-1.0pt \right) \, .
  \end{align*}
  Lemma~\ref{l:h^0} shows that
  $h^0 \bigl( S^1 \mathbin{\!\!/\!} L(\lambda) \kern-1.0pt \bigr) = h^0 \bigl(
  S^1 \mathbin{\!\!/\!} L(\mu) \kern-1.0pt \bigr) + h^0 \bigl( S^1
  \mathbin{\!\!/\!} L(1) \kern-1.0pt \bigr)$ and Lemma~\ref{l:dual} shows that
  \[
    \smash{h^0 \Bigl( \Hom_{S} \kern-0.5pt \bigl( L(\lambda), S \bigr) \kern-2.0pt \Bigr)} 
    = \smash{h^0 \Bigl( \Hom_{S} \kern-0.5pt \bigl( L(1), S \bigr) \kern-2.0pt \Bigr)} 
    + \smash{h^0 \Bigl( \Hom_{S} \kern-0.5pt \bigl( L(\mu), S \bigr) \kern-2.0pt \Bigr)} \, .
  \]
  Since \cite{SS23}*{Remark~2.7} shows that $N_{\lambda} = N_{\mu} + N_{(1)}$
  where
  $N_{\nu} = h^0 \kern-0.5pt \smash{\Bigl( \Hom_{S} \kern-0.5pt \bigl( L(\nu),
    \, S^1 \mathbin{\!\!/\!} L(\nu) \kern-1.0pt \bigr) \kern-3.0pt \Bigr)}$
  for any integer partition $\nu$, we deduce that
  \[
    h^0 \Biggl( \kern-2.0pt \Hom_{S} \kern-2.0pt \left(
      \kern-2.0pt L(\mu), 
      \frac{S^1}{L(\mu)} \kern-2.0pt \right) 
    \kern-4.0pt \Biggr) 
    \kern-1.5pt \oplus
    h^0 \Biggl( \kern-2.0pt \Hom_{S} \kern-2.0pt \left(
      \kern-2.0pt L(1), 
      \frac{S^1}{L(1)} \kern-2.0pt \right) 
    \kern-4.0pt \Biggr)
    = h^{0} \Biggl( \kern-2.0pt \Hom_{S} \kern-2.0pt \left(
      \kern-2.0pt L(\lambda), 
      \frac{S^1}{L(\lambda)} \kern-2.0pt \right)
    \kern-4.0pt \Biggr) \, .
  \]
  Combined with Proposition~\ref{p:dim}, we see that the difference between
  the tangent-space dimensions at the two Borel-fixed points is
  \[  
    h^0 \kern-0.5pt \bigl( S^1 \mathbin{\!\!/\!} L(\lambda) \kern-1.0pt \bigr) 
    - h^0 \kern-1.0pt \Bigl( \kern-0.5pt \Hom_S \bigl(  L(\mu), \,
    S^1 \mathbin{\!\!/\!} L(1) \kern-1.0pt \bigr) \kern-3.0pt \Bigr) \kern-1.0pt
    - h^0 \kern-1.0pt \Bigl( \kern-0.5pt \Hom_S \bigl( L(1), \, 
    S^1 \mathbin{\!\!/\!} L(\mu) \kern-1.0pt \bigr) \kern-3.0pt \Bigr) 
    \, . \qedhere
  \]
\end{proof}

To illustrate Lemma~\ref{l:key}, we show that
$\Quot^{\kern0.5pt q} \kern-0.5pt (\mathcal{O}_{\PP^n}^{\! r})$ is singular
when $n \geqslant 2$, $r \geqslant 2$, and $\lambda = (1,1)$.  In particular,
the scheme $\Quot^{\kern0.5pt 2} \kern-0.5pt (\mathcal{O}_{\PP^n}^{\! r})$ is
singular, but the scheme $\Hilb^{2}(\PP^n)$ is smooth for all $n \geqslant 1$;
see the third condition in Remark~\ref{r:Hil}.

\begin{example}
  \label{e:11} 
  Let $s$ be a nonnegative integer and consider the integer partition
  $\lambda \coloneqq (1,1)$. The associate polynomial is
  $q(t) = s \binom{t+n}{n} + 2 \in \QQ[t]$.  For the subpartition
  $\mu \coloneqq (1)$, we have
  \[
    M 
    \coloneqq \Hom_S \bigl( L(\mu), S^1 \mathbin{\!\!/\!} L(1) \kern-1.0pt \bigr) 
    = \Hom_S \bigl( L(1), S^1 \mathbin{\!\!/\!} L(\mu) \kern-1.0pt \bigr) 
    = \Hom_S \bigl( L(\mu), S^1 \mathbin{\!\!/\!} L(\mu) \kern-1.0pt \bigr) \, .
  \]
  Since $\Hilb^1(\PP^n) = \PP^n$, we also have $n = N_{\mu} = h^0(M)$. It follows that
  \[
    h^0 \Bigl( \kern-0.5pt  
    \Hom_S \bigl( L(\mu), \, S^1 \mathbin{\!\!/\!} L(1) \kern-1.0pt \bigr) \kern-3.0pt \Bigr) \kern-1.5pt
    + h^0 \Bigl( \kern-0.5pt 
    \Hom_S \bigl( L(1), \, S^1 \mathbin{\!\!/\!} L(\mu) \kern-1.0pt \bigr) \kern-3.0pt \Bigr) \kern-1.5pt
    = 2n \, . 
  \]  
  On the other hand, Lemma~\ref{l:h^0} gives
  $h^0 \kern-0.5pt \bigl( S^1 \mathbin{\!\!/\!} L(\lambda) \kern-0.5pt \bigr)
  = 2$.  Thus, Lemma~\ref{l:key} establishes that the scheme
  $\Quot^{\kern-0.5pt q} \kern-0.5pt (\mathcal{O}_{\PP^n}^{\! r})$ is singular
  whenever $r \geqslant s+2$ and $n \geqslant 2$.
\end{example}

\begin{remark}
  Experimental evidence, based on computations in \cite{M2}, suggests the
  following more precise statement.  For any integer partition
  $\lambda \coloneqq (\lambda_{1}, \lambda_{2}, \dotsc, \lambda_{e})$, it
  seems that
  \begin{align*}
    h^0  \kern-1.0pt \Biggl( \kern-2.0pt \Hom_S \biggl( 
    \kern-2.0pt L(1), \,
    \frac{S^1}{L(\lambda)} \biggr) 
    \kern-5.0pt \Biggr) \kern-2.0pt
    &= \!
      \begin{cases}
        n+d+1 \!\!\!\!\!\!
        & \text{if $\lambda \!=\! (1^d)$,} \\[-3pt]
        n+d \!\!\!\!\!\!
        & \text{if
          $\lambda \!=\! (\lambda_{1}, \lambda_{2}, \dotsc, \lambda_{e-d},
          1^{d})$ where $e \!>\! d \!>\! 0$ and $\lambda_{e-d}
          \!>\! 1$,} \\[-3pt]
        2 \!\!\!\!\!\!
        & \text{if $\lambda \!=\! (2)$ or $\lambda \!=\! (\lambda_1,2)$ where
          $n \!\geqslant\! \lambda_{1} \!\geqslant\! 2$,} \\[-3pt]
        1 \!\!\!\!\!\! & \text{otherwise.}
        \end{cases}
    \end{align*}
\end{remark}

The subsequent lemma computes the number of global sections of the sheaf
associated to a particular $S$-module.

\begin{lemma}
  \label{l:m=1} 
  For any nonempty integer partition
  $\lambda \coloneqq (\lambda_1, \lambda_2, \dotsc, \lambda_e)$, we have
  \[ 
    h^0 \Bigl( \kern-0.5pt 
    \Hom_{S} \bigl( L(\lambda), \, S^1 \mathbin{\!\!/\!} L(1) \kern-1.0pt \bigr) 
    \kern-3.0pt \Bigr) \kern-1.5pt 
    = n - \lambda_{e} + 1 \, .
  \]
  In particular, when $\lambda$ has a part not equal to $n$, we have
  $\smash{h^0 \kern-0.5pt \Bigl( \kern-0.5pt \Hom_S \bigl( L(\lambda), S^1
    \mathbin{\!\!/\!} L(1) \kern-1.0pt \bigr) \kern-3.0pt \Bigr)} \kern-1.5pt
  \geqslant 2$.
\end{lemma}

\begin{proof}
  The description of the monomial generators of the lexicographic ideal
  $L(\lambda)$ in Remark~\ref{r:gen} implies that the entries in any
  nontrivial relation among them are in the ideal
  $L(1) = \langle x_0, x_1, \dotsc, x_{n-1} \rangle$.  We may compute the
  $\ZZ$-graded $S$-module
  $M \coloneqq \Hom_{S} \bigl( L(\lambda), \, S^1 \mathbin{\!\!/\!} L(1)
  \kern-1.0pt \bigr)$ by applying the functor
  $\Hom_S \bigl( -, S^1 \mathbin{\!\!/\!} L(1) \kern-1.0pt \bigr)$ to a
  presentation matrix of $L(\lambda)$ and taking the kernel. Since
  Remark~\ref{r:gen} shows that the number of minimal generators of
  $L(\lambda)$ is $n - \lambda_e + 1$, it follows that, up to grading,
    \[
        M \cong \bigoplus_{i=1}^{n-\lambda_{e}+1} 
        \kern-2.0pt \bigl( S^1 \mathbin{\!\!/\!} L(1) \kern-1.0pt \bigr) \, ,
    \]  
    However, the associated sheaf $\widetilde{M}$ does not depend on the grading, because $S^1 \mathbin{\!\!/\!} L(1)$ is supported on the point $P \coloneqq [0 \mathbin{:} 0 \mathbin{:} \dotsb \mathbin{:} 0 \mathbin{:} 1] \in \PP^n$. Thus, we see that $\widetilde{M} = \mathcal{O}_P^{n - \lambda_e + 1}$ and $h^0(M) = n - \lambda_e + 1$.

    The standing assumption that $n \geqslant \lambda_{1} \geqslant \lambda_{e}$ gives $n - \lambda_{e} +1 \geqslant 1$. Moreover, this inequality is an equality if and only if $\lambda_{e} = n$ or equivalently $\lambda = (n^{e})$.
\end{proof}

Given this preparatory work, we now prove the principal result of this section.

\begin{proposition} 
    \label{p:sing}
    Let $s$ be a nonnegative integer, let
    $\lambda \coloneqq (\lambda_1, \lambda_{2}, \dotsc, \lambda_{e-1}, 1)$ be
    an integer partition with $e \geqslant 2$ and at least one part equal to
    $1$, and consider the polynomial
    $q(t) \coloneqq s \binom{t+n}{n} + p_{\lambda}(t) \in \QQ[t]$. When
    $n \geqslant 2$, $r \geqslant s+2$, and $\lambda \neq (n^{e-1},1)$, the
    scheme $\Quot^{\kern0.5pt q }\kern-0.5pt (\mathcal{O}_{\PP^n}^{\! r})$ is
    singular.
\end{proposition}

\begin{proof}  
  Consider the nonempty subpartition
  $\mu \coloneqq (\lambda_1, \lambda_2, \dotsc, \lambda_{e-1})$ of $\lambda$.
  Since Example~\ref{e:11} covers the case $\mu = (1)$, we may assume that
  $\mu \neq (1)$.  Hence, the proper inclusion $L(\mu) \subset L(1)$ of ideals
  implies that
  $\Hom_{S} \bigl( S^1 \mathbin{\!\!/\!} L(1), S^1 \mathbin{\!\!/\!} L(\mu)
  \kern-1.0pt \bigr) = 0$.  Applying the functor
  $\Hom_S \bigl( -, S^1 \mathbin{\!\!/\!} L(\mu) \kern-1.0pt \bigr)$ to the
  short exact sequence
  \[
    0 \longrightarrow 
    L(1) \longrightarrow 
    S^1 \longrightarrow
    S^1 \mathbin{\!\!/\!} L(1) \longrightarrow
    0
  \]
  of $\ZZ$-graded $S$-modules proves that $S^1 \mathbin{\!\!/\!} L(\mu)$ is a
  submodule of
  $\Hom_S \bigl( L(1), S^1 \mathbin{\!\!/\!} L(\mu) \kern-1.0pt \bigr)$.  It
  follows that
  \[  
    h^0 \Bigl( \kern-0.5pt 
    \Hom_S \bigl( L(1), \, S^1 \mathbin{\!\!/\!} L(\mu) \kern-1.0pt \bigr) \kern-3.0pt \Bigr) \kern-1.5pt 
    \geqslant h^0 \bigl( S^1 \mathbin{\!\!/\!} L(\mu) \kern-1.0pt \bigr) \, .
  \]
  Since $\mu \neq (n^{e-1})$, Lemma~\ref{l:m=1} gives
  $h^0 \smash{\Bigl( \Hom_S \bigl( L(\mu), S^1 \mathbin{\!\!/\!} L(1)
    \kern-1.0pt \bigr) \kern-3.0pt \Bigr)} \geqslant 2$.  Furthermore,
  Lemma~\ref{l:h^0} shows that
  $h^0 \bigl( S^1 \mathbin{\!\!/\!} L(\lambda) \kern-1.0pt \bigr) = h^0 \bigl(
  S^1 \mathbin{\!\!/\!} L(\mu) \kern-1.0pt \bigr) + 1$, so we obtain
  \[  
    h^0 \Bigl( \kern-0.5pt 
    \Hom_S \bigl( L(\mu), \, S^1 \mathbin{\!\!/\!} L(1) \kern-1.0pt \bigr)
    \kern-3.0pt \Bigr) \kern-1.5pt
    + h^0 \Bigl( \kern-0.5pt \Hom_S \bigl( L(1), \, S^1 \mathbin{\!\!/\!}
    L(\mu) \kern-1.0pt \bigr) \kern-3.0pt \Bigr) \kern-1.5pt \geqslant 2 + h^0
    \kern-0.5pt \bigl( S^1 \mathbin{\!\!/\!} L(\mu) \kern-1.0pt \bigr)
    \kern-0.5pt
    > h^0 \kern-0.5pt \bigl( S^1 \mathbin{\!\!/\!} L(\lambda) \kern-1.0pt
    \bigr) \, .
  \]
  Thus, Lemma~\ref{l:key} establishes that the scheme
  $\Quot^{\kern0.5pt q} \kern-0.5pt (\mathcal{O}_{\PP^n}^{\! r})$ is singular.
\end{proof}

\begin{example}
  The Koszul complex gives a surjective morphism
  $\smash{\mathcal{O}^{\! \binom{n+1}{2}}_{\mathbb{P}^{n}}} \to
  \Omega_{\mathbb{P}^{n}}(2)$.  When $n = 3$, the Hilbert polynomial of the
  twisted cotangent bundle $ \Omega_{\mathbb{P}^{3}}(2)$ is
  \[
    q(t) 
    = \tfrac{1}{2}t^{3} + 4 t^{2} + \tfrac{19}{2} t + 6
    = 3 \tbinom{t+3}{3} +  p_{\lambda}(t) \, ,
  \]
  where $\lambda \coloneqq (3^2, 2^2, 1^5)$. When $n = 3$, $r = 6$, $s = 3$,
  and $\lambda = (3^2, 2^2, 1^5)$, Proposition~\ref{p:sing} shows that the
  Quot scheme $\Quot^{\kern0.5pt q} \kern-0.5pt (\mathcal{O}_{\PP^3}^{\! 6})$,
  which contains the flat deformations of $\Omega_{\mathbb{P}^{3}}(2)$, is
  singular.
    
  For all $3 \leqslant n \leqslant 9$, calculations in \cite{M2} indicate that
  the multiplicity $a_1$ for the part of $1$ in the integer partition
  associated to the Hilbert polynomial of $ \Omega_{\mathbb{P}^{n}}(2)$ grows
  exponentially; the limited data suggests
  $1 + \lfloor \log_{10} a_1 \rfloor \in O(2.17^{n})$.  In particular, it
  appears that the Quot schemes, containing a point corresponding to
  $\Omega_{\mathbb{P}^{n}}(2)$, are singular for all $n \geqslant 3$.
\end{example}

\section{Classifying Smooth Quot Schemes}
\label{s:main}

\noindent
The purpose of this section is to prove Theorem~\ref{t:main}.  By
Lemma~\ref{l:smooth}, we may assume that the base scheme is
$\Spec(\kk)$. Corollary~\ref{c:easy} already demonstrates that the scheme
$\Quot^{\kern0.5pt q} \kern-0.5pt (\mathcal{O}_{\PP^n}^{\! r})$ is nonsingular
when the integer partition
$\lambda \coloneqq (\lambda_1, \lambda_2, \dotsc, \lambda_e)$ satisfies
$\lambda_{e} \geqslant 2$, $\lambda = (1)$, or $\lambda = ()$. As a partial
converse, Proposition~\ref{p:sing} establishes that the scheme
$\Quot^{\kern0.5pt q} \kern-0.5pt (\mathcal{O}_{\PP^n}^{\! r})$ is singular
when $r \geqslant s+2$ and the integer partition $\lambda$ has at least one
part equal to $1$ but is not equal to $(n^{e-1}, 1)$ for some $e \geqslant 2$.

Among the lingering possibilities, we first suppose that $r \geqslant s + 2$
and $\lambda = (n^{e-1}, 1)$ for some $n \geqslant 2$ and $e \geqslant 2$. The
standard graded polynomial ring $S \coloneqq \kk[x_0, x_1, \dotsc, x_n]$
continues to be the homogeneous coordinate ring of $\PP^n$.  For notational
brevity, set
$L(2) \coloneqq L\bigl( \kern-0.5pt (2) \kern-0.5pt \bigr) = \langle x_0, x_1,
\dotsc, x_{n-2} \rangle$.

\begin{lemma}
  \label{l:2b}
  Let $s$ be a nonnegative integer, let $\lambda \coloneqq (n^{e-1}, 1)$ be an
  integer partition with $e \geqslant 2$, and consider the polynomial
  $q(t) \coloneqq s \binom{t+n}{n} + p_{\lambda}(t) \in \QQ[t]$. Assume that
  $r \geqslant s+2$ and regard $\mu \coloneqq (n^{e-1})$ as a subpartition of
  $\lambda$. If $\lambda \neq (2,2,1)$, then there are exactly two Borel-fixed
  submodules in $S^{r}$ whose quotients have Hilbert polynomial $q$:
  \begin{align*}
    & S^{r-s-1} \oplus L(\lambda) \oplus 0^{s} \subseteq S^{r}
    && \text{and} 
    && S^{r-s-2} \oplus L(1) \oplus L(\mu) \oplus 0^{s} \subseteq S^{r} \, . 
  \end{align*}
  Otherwise $\lambda = (2,2,1)$ and there are exactly three Borel-fixed
  submodules in $S^{r}$ whose quotients have Hilbert polynomial $q$: the above
  two and $S^{r-s-2} \oplus L(2) \oplus L(2) \oplus 0^{s} \subseteq S^{r}$.
\end{lemma}

\begin{proof}
  Lemma~\ref{l:sum} demonstrates that the Hilbert polynomial $p_{\lambda}$ can
  be expressed as a nontrivial sum of Hilbert polynomials in at most two ways:
  $p_{\lambda} = p_{\mu} + p_{(1)}$ or, when $n = 2$ and $e = 3$,
  $p_{\lambda} = p_{(2)} + p_{(2)}$.  By \cite{Sta20}*{Theorem~1.1}, there is
  a unique saturated Borel-fixed submodule with Hilbert polynomial
  $p_{\lambda}$, $p_{\mu}$, $p_{(1)}$, or $p_{(2)}$.  Therefore, for each
  expression, there is a unique Borel-fixed submodule in $S^r$ whose
  components have the corresponding Hilbert polynomials.
\end{proof}

To handle this penultimate case, we need one more straightforward calculation.

\begin{lemma}
  \label{l:presM}   
  Assume that $n \geqslant 2$ and $e \geqslant 2$.  Consider the integer
  partition $\mu \coloneqq (n^{e-1})$.  For the $\ZZ$-graded $S$-module
  $M \coloneqq \Hom_S \bigl( L(1), \, S^1 \mathbin{\!\!/\!}  L(\mu)
  \kern-1.0pt \bigr)$, we have%
  \vspace*{-0.35em}
  \begin{align*}
    M &\cong \begin{cases}
      \left( \! \dfrac{S^1}{L(\mu)} \! \right) \kern-3.0pt (1) 
      & \hspace{-5pt} 
      \text{if $n = 2$ and $e = 2$,} \\[12pt]
      \dfrac{\kern-0.5pt \bigl( S^1 \oplus S^1(-e+3) \kern-1.0pt \bigr)}%
      {\operatorname{Im} \begin{bsmallmatrix} 
          \smash[t]{x_0^{e-2}} & 0 \\
          -x_1 & x_0 \\
        \end{bsmallmatrix}} 
      & \hspace{-5pt} 
      \text{if $n = 2$ and $e \geqslant 3$,} \\[2pt]
      \dfrac{S^1}{L(\mu)} 
      & \hspace{-5pt} 
      \text{if $n \geqslant 3$;} \\[-2pt]
    \end{cases} 
    \intertext{and} \\[-30pt]
    h^0({M}) &= \begin{cases}
      2 & \hspace{-5pt}
      \text{if $n = 2$ and $2 \leqslant e \leqslant 3$,} \\[-2pt]
      1 & \hspace{-5pt}
      \text{otherwise.}
    \end{cases}
  \end{align*}
\end{lemma}

\begin{proof} 
  Recall that $L(1) = \langle x_0, x_1, \dotsc, x_{n-1} \rangle$ and
  $L(\mu) = \langle x_0^{e-1} \rangle$; see Remark~\ref{r:gen}.  The minimal
  free resolution of the $\ZZ$-graded $S$-module $S^1 \mathbin{\!\!/\!} L(1)$
  is a Koszul complex, so the $S$-module $L(1)$ has projective dimension $n-1$
  and a free presentation of the form
  $\smash{S^{\binom{n}{2}}(-2)} \longrightarrow S^n(-1) \longrightarrow L(1)
  \longrightarrow 0$.  Applying the functor $\Hom_{S}( -, S^1)$, we see that
  $\Hom_S \kern-0.5pt \bigl( L(1), S^1 \kern-0.5pt \bigr) \cong S^1$ and
  $\Ext^i_{S} \kern-0.5pt \bigl( L(1), S^1 \kern-0.5pt \bigr) = 0$ for all
  $1 \leqslant i \leqslant n-2$, because the Koszul complex is self-dual.

  Suppose that $n \geqslant 3$.  It follows that
  $\Ext^1_{S} \kern-0.5pt \bigl( L(1), S^1 \kern-0.5pt \bigr) = 0$ and
  $\Ext^1_{S} \kern-0.5pt \bigl( L(1), S^1(-e+1) \kern-1.0pt \bigr) = 0$.
  Hence, applying the functor
  $\Hom_{S} \kern-0.5pt \bigl( L(1), - \kern-0.5pt \bigr)$ to the short exact
  sequence
  \[ 
    0 \longrightarrow S^1(-e+1) \xrightarrow{
      \;\; \begin{bsmallmatrix*}
        x_0^{e-1} \\
      \end{bsmallmatrix*} \;\;}
    S^1 \longrightarrow S^1 \mathbin{\!\!/\!} L(\mu) \longrightarrow 0 
  \]
  of $\ZZ$-graded $S$-modules shows that
  $M \coloneqq \Hom_{S} \kern-0.5pt \bigl( L(1), S^1 \mathbin{\!\!/\!} L(\mu)
  \kern-0.5pt \bigr) \cong S^{1} \mathbin{\!\!/\!}  L(\mu)$.
    
  Suppose that $n = 2$. Applying the functor
  $\Hom_S \kern-0.5pt \bigl(-, S^1 \mathbin{\!\!/\!} L(\mu) \kern-0.5pt
  \bigr)$ to the free presentation yields
  \[
    0 
    \longrightarrow \Hom_S \!\left( \! L(1), \frac{S^1}{L(\mu)} \!\right) 
    \longrightarrow  \bigoplus_{i=1}^{2} \left( \! \frac{S^1}{L(\mu)} (-1) \!\!\right)
    \xrightarrow{\;\; \begin{bsmallmatrix*}
        -x_1 & x_0 \\    
      \end{bsmallmatrix*} \;\; }
    \frac{S^1}{L(\mu)} (-2) \, .
  \]
  When $e \geqslant 3$, it follows that 
  \[
    M = \Hom_S \!\left( \! L(1), \frac{S^1}{L(\mu)} \!\right) 
    \cong \left( \operatorname{Im} \begin{bmatrix*}
        x_0 & 0 \\
        x_1 & x_0^{e-2} \\
      \end{bmatrix*}
      + \operatorname{Im} \begin{bmatrix*}
        \smash[t]{x_0^{e-1}} & 0 \\
        0 & \smash[t]{x_0^{e-1}} \\
      \end{bmatrix*} \right) \!
    \bigg/ \operatorname{Im} \begin{bmatrix*}
      \smash[t]{x_0^{e-1}} & 0 \\
      0 & \smash[t]{x_0^{e-1}} \\
    \end{bmatrix*} \, . 
  \]
  Finding a minimal presentation for this subquotient gives
  $M \cong \kern-0.5pt \bigl( S^1 \oplus S^1(-e+3) \kern-1.0pt \bigr)
  \smash{\Big/ \operatorname{Im} \begin{bsmallmatrix*}
      \smash[t]{x_0^{e-2}} & 0 \\
      -x_1 & x_0 \\
    \end{bsmallmatrix*}}$.  Similarly, when $e = 2$, we have 
  \[
    M = \Hom_S \!\left( \! L(1), \frac{S^1}{L(\mu)} \!\right) 
    \cong \left( \operatorname{Im} \begin{bmatrix*}
        0 \\
        1 \\
      \end{bmatrix*}
      + \operatorname{Im} \begin{bmatrix*}
        x_0 & 0 \\
        0 & x_0 \\
      \end{bmatrix*} \right) \!
    \bigg/ \operatorname{Im} \begin{bmatrix*}
      x_0 & 0 \\
      0 & x_0 \\
    \end{bmatrix*}
    \cong
    \left( \! \dfrac{S^1}{L(\mu)} \! \right) \kern-3.0pt (1) 
    \, . 
  \]
    
  Finally, Lemma~\ref{l:h^i} implies that $h^{0}(M) = \dim_{\kk} M_{0}$
  because the projective dimension of $M$ is $1$. Thus, the second part of the
  conclusion follows from the first.
\end{proof}

Given these preliminaries, we can understand this special case.

\begin{proposition}
  \label{p:n^e1}
  Let $s$ be a nonnegative integer, let $\lambda \coloneqq (n^{e-1}, 1)$ be an
  integer partition with $e \geqslant 2$, and consider the polynomial
  $q(t) \coloneqq s \binom{t+n}{n} + p_{\lambda}(t) \in \QQ[t]$. Assume that
  $n \geqslant 2$ and $r \geqslant s+2$. The dimensions of the tangent spaces
  to $\Quot^{\kern0.5pt q} \kern-0.5pt (\mathcal{O}_{\PP^n}^{\! r})$ at the
  Borel-fixed points are equal if and only if we have $\lambda\neq (2,1)$ and
  $\lambda \neq (2,2,1)$. In particular, the scheme
  $\Quot^{\kern0.5pt q} \kern-0.5pt (\mathcal{O}_{\PP^n}^{\! r})$ is
  nonsingular if and only if $\lambda \neq (2,1)$ and $\lambda \neq (2,2,1)$.
\end{proposition}

\begin{proof} 
  Let $\mu \coloneqq (n^{e-1})$ be a nonempty subpartition of $\lambda$.
  Lemma~\ref{l:2b} identifies the Borel-fixed points in
  $\Quot^{\kern0.5pt q} \kern-0.5pt (\mathcal{O}_{\PP^n}^{\!
    r})$. Theorem~\ref{t:smooth} shows that the lexicographic point on
  $\Quot^{\kern0.5pt q} \kern-0.5pt (\mathcal{O}_{\PP^n}^{\! r})$ is
  nonsingular. When the Quot scheme has exactly two Borel-fixed points, one
  can prove that it is nonsingular by showing that the dimensions of the
  tangent spaces at these points are equal---the dimension of the tangent
  space at a point in family is an upper-semicontinuous function and there is
  a one-parameter family connecting any point to a Borel-fixed point.
  Therefore, it suffices to prove that the dimensions of the tangent spaces to
  $\Quot^{\kern0.5pt q} \kern-0.5pt (\mathcal{O}_{\PP^n}^{\! r})$ at the
  Borel-fixed points corresponding to the submodules
  $S^{r-s-1} \oplus L(\lambda) \oplus 0^{s} \subseteq S^{r}$ and
  $S^{r-s-2} \oplus L(1) \oplus L(\mu) \oplus 0^{s} \subseteq S^{r}$ are equal
  if and only if we have $\lambda \neq (2,1)$ and $\lambda \neq (2,2,1)$.
    
  We establish this equivalence by using Lemma~\ref{l:key}. First,
  Lemma~\ref{l:h^0} gives
  $h^0 \kern-0.5pt \bigl( S^1 \mathbin{\!\!/\!} L(\lambda) \kern-0.5pt \bigr)
  = 2$. Second, as $L(\mu) = \smash{\langle x_0^{e-1} \rangle}$, it follows
  that
  \[
    M \coloneqq \Hom_S \bigl( L(\mu), \, S^1 \mathbin{\!\!/\!} L(1) \kern-1.0pt \bigr) \kern-0.5pt 
    \cong \Hom_S \bigl( S^1(-e+1), \, S^1 \mathbin{\!\!/\!} L(1) \kern-1.0pt \bigr) \kern-0.5pt 
    \cong \bigl( S^1 \mathbin{\!\!/\!} L(1) \kern-1.0pt \bigr) \!(e-1)
  \]
  The minimal free $\ZZ$-graded resolution of the $S$-module
  $S^1 \mathbin{\!\!/\!} L(1) = S^1 \mathbin{\!\!/\!} \langle x_0, x_1,
  \dotsc, x_{n-1} \rangle$ is a Koszul complex having length $n$, so
  $\Ext^{n}_{S}(M, S^1) \cong \bigl( \kern-0.5pt S^1 \mathbin{\!\!/\!} L(1)
  \kern-0.5pt \bigr)(e-n-1)$ and $\Ext^{n+1}_{S}(M, S^1) = 0$. By
  Lemma~\ref{l:h^i}, we deduce that $h^0(M) = 1$.  Third, Lemma~\ref{l:presM}
  establishes that
  \[
    h^0 \Bigl( \Hom_S \bigl( L(1), \, 
    S^1 \mathbin{\!\!/\!} L(\mu) \kern-1.0pt \bigr) \kern-3.0pt \Bigr) 
    = \begin{cases}
      2 & \text{if $\lambda = (2,1)$ or $\lambda = (2,2,1)$,} \\[-2pt]
      1 & \text{otherwise.}
    \end{cases}
  \]
  Altogether, we obtain
  \[
    h^0 \bigl( S^1 \mathbin{\!\!/\!} L(\mu) \kern-1.0pt \bigr) 
    - h^0(M)
    - h^0 \Bigl( \Hom_S \bigl( L(1), \, 
    S^1 \mathbin{\!\!/\!} L(\mu) \kern-1.0pt \bigr) \kern-3.0pt \Bigr) 
    = 
    \begin{cases}
      -1 
      & \text{if $\lambda = (2,1)$ or $\lambda = (2,2,1)$,} \\[-2pt]
      0  
      & \text{otherwise.}
    \end{cases}
  \]
  so Lemma~\ref{l:key} completes the proof.
\end{proof}

The ($r= s+1$)-case is all that remains. The proof of the next proposition is
delightfully similar to Lemma~\ref{l:key}.

\begin{proposition}
  \label{p:s+1}
  Let $s$ be a nonnegative integer, let
  $\lambda \coloneqq (\lambda_{1}, \lambda_{2}, \dotsc, \lambda_{e})$ be an
  integer partition, and consider the polynomial
  $q(t) \coloneqq s \binom{t+n}{n} + p_{\lambda}(t) \in \QQ[t]$. When
  $r = s + 1$, the scheme
  $\Quot^{\kern0.5pt q} \kern-0.5pt (\mathcal{O}_{\PP^n}^{\! r})$ is singular
  if and only if the scheme
  $\Hilb^{\kern0.5pt p_{\lambda}} \kern-0.5pt (\mathcal{O}_{\PP^n}^{\! r})$ is
  singular.
\end{proposition}

\begin{proof}
  Let $a_{n}$ denote the multiplicity of the part $n$ in the integer partition
  $\lambda$. The hypothesis $r = s+1$ implies that every Borel-fixed point in
  $\Quot^{\kern0.5pt q} \kern-0.5pt (\mathcal{O}_{\PP^n}^{\! r})$ corresponds
  to a submodule $I \oplus 0^{r-1} \subseteq S^{r}$, where $I$ is a
  Borel-fixed ideal such that the Hilbert polynomial of
  $S^1 \mathbin{\!\!/\!} I$ is
  $p_{\lambda}(t) \coloneqq \smash{\sum_{i=1}^{e}
    \binom{t+\lambda_i-i}{\lambda_i-1}} \in \QQ[t]$.  The tangent space to the
  Quot scheme at this point is the space of global sections of the sheaf
  associated to the $\mathbb{Z}$-graded $S$-module
  \[
    \Hom_S ( I, S^1 \mathbin{\!\!/\!} I \oplus S^{r-1})
    = \Hom_S( I, S^1 \mathbin{\!\!/\!} I )
    \oplus \kern-1.0pt \left( 
      \bigoplus_{i=1}^{r-1} \Hom_S ( I, S) 
      \kern-4.0pt \right) \, . 
  \]
  Since the Hilbert polynomial of the quotient $S^1 \mathbin{\!\!/\!} I$ is
  $p_{\lambda}(t)$, the greatest common divisor of the minimal monomial
  generators of $I$ has degree $a_{n}$; compare with
  \cite{Fog68}*{Theorem~1.4}. Hence, Lemma~\ref{l:dual} shows that
  \[
    h^0 \kern-0.5pt \Bigl( \Hom_S \bigl( L(\lambda), S \bigr) \kern-2.5pt \Bigr) 
    = \binom{n+a_n}{n}       
    = h^0  \bigl( \Hom_S ( I, S) \kern-1.0pt \bigr) \, , 
  \]
  where $L(\lambda)$ is the lexicographic ideal; see
  Remark~\ref{r:gen}. Combined with Proposition~\ref{p:dim}, we see that the
  difference between the tangent-space dimensions to the Quot scheme at the
  lexicographic point and another Borel-fixed point is
  \[
    h^0 \kern-1.0pt \Bigl( \kern-0.5pt \Hom_S \bigl(  L(\lambda), \,
    S^1 \mathbin{\!\!/\!} L(\lambda) \kern-1.0pt \bigr) \kern-3.0pt \Bigr) \kern-1.0pt
    - h^0 \bigl(  \Hom_S (I, \, S^1 \mathbin{\!\!/\!} I) \kern-1.0pt \bigr) 
    \, .
  \]
  This expression is also the difference between the tangent space dimensions
  to the Hilbert scheme
  $\Hilb^{\kern0.5pt p_{\lambda}} \kern-0.5pt (\mathcal{O}_{\PP^n}^{\! r})$ at
  its lexicographic point and the Borel-fixed point corresponding to the ideal
  $I$. Theorem~\ref{t:smooth} shows that the lexicographic points on both
  $\Quot^{\kern0.5pt q} \kern-0.5pt (\mathcal{O}_{\PP^n}^{\! r})$ and
  $\Hilb^{\kern0.5pt p_{\lambda}} \kern-0.5pt (\mathcal{O}_{\PP^n}^{\! r})$
  are nonsingular. Since
  $\Quot^{\kern0.5pt q} \kern-0.5pt (\mathcal{O}_{\PP^n}^{\! r})$ and
  $\Hilb^{\kern0.5pt p_{\lambda}} \kern-0.5pt (\mathcal{O}_{\PP^n}^{\! r})$
  are path-connected, having a Borel-fixed point whose tangent-space dimension
  is different than the lexicographic point is equivalent to being singular.
\end{proof}

\begin{remark}
  \label{r:bun}
  When $r = s + 1$, one may generalize Proposition~\ref{p:s+1} to show that
  $\Quot^{\kern0.5pt q} \kern-0.5pt (\mathcal{O}_{\PP^n}^{\! r})$ is locally a
  product of $\mathbb{P}^m$, where $m \coloneqq r \binom{n+a_{n}}{a_{n}}$, and
  $\Hilb^{\kern0.5pt p_{\lambda}} \kern-0.5pt (\mathcal{O}_{\PP^n}^{\!
    r})$. By analogy with \cite{Fog68}*{Theorem~1.4}, one might hope that this
  a global product; compare with Lemma~\ref{l:morp}.
\end{remark}

Since the numerical conditions on the integer partition
$\lambda \coloneqq (\lambda_1, \lambda_2, \dotsc, \lambda_e)$ that
characterize precisely when
$\Hilb^{\kern0.5pt p_{\lambda}} \kern-0.5pt (\PP^n)$ is smooth are known, our
last case can be rephrased as follows.

\begin{remark}[Combinatorial conditions on $\lambda$ in the ($r=s+1$)-case]
  \label{r:Hil}
  Using \cite{SS23}*{Theorem~A}, the ($r=s+1$)-case in Theorem~\ref{t:main} becomes
  \begin{compactitem}[\phantom{W} $\bullet$]
  \item $r=s+1$ and one of the following conditions holds:
    \begin{compactenum}[\phantom{W} \upshape 1.]
    \item $n = 2$;
    \item $\lambda_e \geqslant 2$;
    \item $\lambda =(n^{e-2}, \lambda_{e-1}, 1)$ where
      $n \geqslant \lambda_{e-1} \geqslant 1$ and $e \geqslant 2$;
    \item
      $\lambda = (n^{e-d-3}, \lambda_{e-\smash{d}-2}^{\kern-1.0pt d+2}, 1)$
      where $n-1 \geqslant \lambda_{e-d-2} \geqslant 3$ and
      $e-3 \geqslant d \geqslant 0$ and;
    \item $\lambda = (n^{e-d-5}, 2^{d+4}, 1)$ where $n \geqslant 3$ and
      $e-5 \geqslant d \geqslant 0$;
    \item $\lambda = (n^{e-3}, 1^3)$ where $n \geqslant 3$ and
      $e \geqslant 3$;
    \item $\lambda = (1)$ or $\lambda = ()$.
    \end{compactenum}
  \end{compactitem}
  The ($\lambda \!=\! (n \!+\! 1)$)-case in \cite{SS23}*{Theorem~A} is
  subsumed by the ($r=s$)-case in Theorem~\ref{t:main}.
\end{remark}

We now prove the main theorem of the paper.

\begin{proof}[Proof of Theorem~\ref{t:main}]   
  Lemma~\ref{l:smooth} reduces the study of smoothness to the base scheme
  being $\Spec(\kk)$.  By Proposition~\ref{p:nonempty}, the scheme
  $\Quot^{\kern0.5pt q} \kern-0.5pt (\mathcal{O}_{\PP^n}^{\! r})$ is nonempty
  if and only if there exist a nonnegative integer $s$ and an integer
  partition $\lambda \coloneqq (\lambda_1, \lambda_2, \dotsc, \lambda_e)$ such
  that $r > s$, $n \geqslant \lambda_1$, and
  \[
    q(t)
    = s \, \binom{t+n}{n} + p_{\lambda}(t)
    = s \, \binom{t+n}{n} + \sum_{i=1}^{e} \binom{t+\lambda_i-i}{\lambda_i-1} 
    \, ;
  \]
  or $r = s$, $\lambda = ()$, and $q(t) = r \, \binom{t+n}{n}$.  In the
  ($r=s)$-case, the only quotient is
  $\mathcal{F} = \mathcal{O}_{\PP^n}^{\!r}$, so the Quot scheme is just a
  single point and, thereby, smooth. When $r = s+1$, Proposition~\ref{p:s+1}
  establishes the desired characterization.

  The ($r \geqslant s+2$)-case remains. When $\lambda_{e} \geqslant 2$,
  $\lambda = (1)$, or $\lambda = ()$, Corollary~\ref{c:easy} demonstrates that
  $\Quot^{\kern0.5pt q} \kern-0.5pt (\mathcal{O}_{\PP^n}^{\! r})$ is
  nonsingular.  When $n \geqslant 2$, Proposition~\ref{p:sing} establishes
  that $\Quot^{\kern0.5pt q} \kern-0.5pt (\mathcal{O}_{\PP^n}^{\! r})$ is
  singular whenever the integer partition $\lambda$ has at least one part
  equal to $1$ and is not equal to $(n^{e-1}, 1)$ for some $e \geqslant
  2$. Furthermore, when $n \geqslant 2$, $e \geqslant 2$, and
  $\lambda = (n^{e-1}, 1)$, Proposition~\ref{p:n^e1} shows that
  $\Quot^{\kern0.5pt q} \kern-0.5pt (\mathcal{O}_{\PP^n}^{\! r})$ is
  nonsingular exactly when $\lambda \neq (2,1)$ and $\lambda \neq
  (2,2,1)$. Finally, when $n=1$, \cite{Str87}*{Theorem 2.1} proves that any
  Quot scheme is nonsingular. This exhausts all possible integer partitions,
  so the proof is complete.
\end{proof}

\section{Two Singular Examples}
\label{s:exa}

\noindent
This final section describes, in detail, two singular Quot schemes.  Both
resemble the Hilbert scheme compactifying twisted cubic curves~\cite{PS85} by
having two smooth rational irreducible components of unequal dimensions whose
intersection is also smooth and rational.  These Quot schemes also explain the
discrepancy between second and third condition in the ($r\geqslant s+2$)-case
of Theorem~\ref{t:main}.  The standard graded polynomial ring
$S \coloneqq \kk[x_0,x_1, \dotsc, x_n]$ continues to denote the homogeneous
coordinate ring of $\PP^n$.

To analyze examples, we capitalize on a specific construction of the Quot
scheme $\Quot^{\kern0.5pt q} \kern-0.5pt (\mathcal{O}_{\PP^n}^{\! r})$ that
generalizes the Gotzmann approach for Hilbert schemes; compare with
\cite{IK99}*{Proposition~C.28}.  We assume that the degree of the Hilbert
polynomial $q$ is less than $n$; this is equivalent to having $s = 0$ by
Proposition~\ref{p:nonempty}.  Let $e$ be the number of parts in the integer
partition $\lambda$ associated to $q$ (also known as the \emph{Gotzmann
  number}).  Set $k \coloneqq q(e)$ to be the nonnegative integer obtained by
evaluating the polynomial $q$ at $e$. The $e$th homogeneous part $S^r_{e}$ of
the free $S$-module of rank $r$ is a $\kk$-vector space of dimension
$\smash{d \coloneqq r \binom{e+n-1}{n-1}}$.  We may regard a point $V$ in the
Grassmannian $\Gr(k, d)$ as a linear subspace in $S^r_e$ of dimension
$d-k$. Multiplication by all linear forms in $S$ produces the linear subspace
$S_1 \cdot V$ in $S^{r}_{e+1}$ and the \emph{corank} of $S_1 \cdot V$ is the
dimension of the quotient $S^{r}_{e+1} \mathbin{\!\!/\!}  (S_1 \cdot V)$.  For
any positive integers $n$ and $r$, the Quot scheme
$\smash{\Quot^{\kern0.5pt q} \kern-0.5pt (\mathcal{O}_{\PP^n}^{\! r})}$ is
equal to the closed subscheme
\[
  \bigl\{ V \in \Gr(k, d) \mathrel{\big|} \operatorname{corank}(S_1 \cdot V) =
  q(e+1) \bigr\}
\]  
in the Grassmannian $\Gr(k, d)$; see \cite{Del16}*{Theorem~5.1}.

This construction produces local equations for
$\smash{\Quot^{\kern0.5pt q} \kern-0.5pt (\mathcal{O}_{\PP^n}^{\! r})}$ in a
neighbourhood of a Borel-fixed point.  Each point $V$ in $\Gr(k,d)$ is
represented by the span of $d-k$ linearly independent vectors in $S^r_{e}$.
Choosing monomials as a $\kk$-vector space basis for $S^r_{e}$, the linear
subspace $V$ is the row span of a $\bigl( \! (d-k) \times d \bigr)$-matrix.
The standard affine charts on $\Gr(k,d)$ are indexed by a list of $d-k$
columns or, equivalently, $d-k$ monomials in $S^{r}_{e}$.  By listing the
monomial generators of the $e$th truncation of a submodule corresponding to a
Borel-fixed point, we obtain an affine neighbourhood of this Borel-fixed
point.  For the $\bigl( \! (d-k) \times d \bigr)$-matrices in an affine chart,
the submatrix determined by the listed columns is the identity matrix
$\mathbf{I}_{d-k}$ and the complementary entries correspond to a point in
$\AA^{k(d-k)}$.  The condition $\operatorname{corank}(S_1 \cdot V) = q(e+1)$
thereby identifies
$\smash{\Quot^{\kern0.5pt q} \kern-0.5pt (\mathcal{O}_{\PP^n}^{\! r})}$ with a
determinantal subscheme.  In an affine chart, the Quot scheme
$\smash{\Quot^{\kern0.5pt q} \kern-0.5pt (\mathcal{O}_{\PP^n}^{\! r})}$ is
defined by the ideal of all $\bigl(q(e+1)+1 \! \bigr)$-minors of the
$\bigl( (n+1)(d-k) \times r \binom{e+n}{n} \bigr)$-matrix corresponding to the
linear subspace $S_1 \cdot V$; compare with \cite{IK99}*{Proposition~C.30} or
\cite{HS04}*{\S4}.  To work efficiently with this affine subscheme, it is
important to eliminate the many variables that appear as a linear term in some
generator of this nonhomogeneous ideal.  Fortuitously, the required affine
subschemes are just small enough to be effectively computed in our
examples---we are able to decompose the affine subschemes into their
irreducible components.

\begin{example}
  \label{e:21}
  Let $Q^{(2,1)} \coloneqq \smash{\Quot^{t+2}(\mathcal{O}_{\PP^{2}}^{\!2})}$
  denote the Quot scheme parametrizing quotients of the trivial rank-$2$
  vector bundle on the projective plane having Hilbert polynomial $t+2$.  This
  Hilbert polynomial corresponds to the integer partition
  $\lambda \coloneqq (2,1)$, because
  $t+2 = \smash{\binom{t+1}{1} + \binom{t-1}{0}}$.  The Quot scheme
  $Q^{(2,1)}$ has precisely two Borel-fixed points corresponding to the
  cokernels of the homogeneous $S$-linear maps
  \begin{align*}
    S^2
    &\xleftarrow{{ \quad
      \mathbf{B}_0 \coloneqq \begin{bmatrix}
        1 & 0 & 0  \\
        0 & x_0^2 & x_0^{} \, x_1^{}  \\
      \end{bmatrix} \quad}}
    S \oplus S^2(-2) 
    && \text{and}
    & S^2 
    &\xleftarrow{{ \quad
      \mathbf{B}_1 \coloneqq \begin{bmatrix}
        x_0^{} & x_1^{} & 0 \\
        0 & 0 & x_0^{} \\
      \end{bmatrix} \quad}}
    S^3(-1)  \, .
  \end{align*}
  The lexicographic point corresponds to $\coker \mathbf{B}_0$.  By
  calculating
  $h^0 \bigl( \Hom_S( \image \mathbf{B}_{i}, \coker \mathbf{B}_{i}) \!\bigr)$
  in \emph{Macaulay2}~\cite{M2} for all $0 \leqslant i \leqslant 1$, we see
  that the tangent spaces at the points determined by $\coker \mathbf{B}_0$
  and $\coker \mathbf{B}_1$ have dimension $6$ and $7$ respectively.  It
  follows that $Q^{(2,1)}$ is singular.  A local analysis in a neighbourhood
  of the Borel-fixed point shows that the point corresponding to
  $\coker \mathbf{B}_1$ lies on two distinct irreducible components: the
  transverse intersection of two affine spaces having dimensions $6$ and $5$.
  The lexicographic point $\coker \mathbf{B}_0$ lies on a unique irreducible
  component because it is a smooth point; see Theorem~\ref{t:smooth}.  We
  conclude that both irreducible components of $Q^{(2,1)}$ are smooth.

  \begin{figure}[!ht]
    \addtocounter{lemma}{1}
    \centering
    \begin{tikzpicture}[xscale=1.0, yscale=0.5]
      \path[draw, very thick, name path=border1] (0,0) to[out=-10,in=130] (6,-2);
      \path[draw, very thick, name path=line1] (6,-2) to[out=0,in=200] (12,1);
      \path[draw, very thick, name path=border2] (12,1) to[out=130,in=-10] (5.5,3.7);
      \path[draw, very thick, name path=line2] (5.5,3.7) to[out=200,in=0] (0,0);
      \shade[left color=gray!20,right color=gray!70] 
      (0,0) to[out=-10,in=130] (6,-2) to[out=0,in=200]
      (12,1) to[out=130,in=-10] (5.5,3.7) to[out=200,in=0] cycle;
      \path[draw, very thick, name path=border3] (-1,-4) to[out=20,in=220] (3,3);
      \path[draw, very thick, name path=border6] (3,3) to[out=10,in=140] (9,1);
      \path[draw, very thick, name path=border4] (9,1) to[out=220,in=20] (6,-7);
      \path[draw, very thick, name path=border5] (6,-7) to[out=140,in=10] (-1,-4);
      \path[name intersections={of=border3 and line2, by={a}}];
      \path[name intersections={of=border4 and line1, by={b}}];
      \shade[top color=gray!10,bottom color=gray!90,opacity=.30] 
      (-1,-4) to[out=20,in=220] (3,3) to[out=10,in=140] 
      (9,1) to[out=220,in=20] (6,-7) to[out=140,in=10] (-1,-4);
      \path[draw, ultra thick, densely dotted, name path=int] (a) to[out=-10,in=130] (b);
      \path[name path=line3] (6,-2) -- (6,2);
      \path[name intersections={of=int and line3, by={c}}];
      \node at (1.1,-3.3) {$5$-dim};
      \node at (10.4,1.3) {$6$-dim};
      \node at (0,2) (label) {$4$-dim};
      \draw[-Latex, thick, shorten <= -3pt] (label) -- (a);
      \node at (c) (B1) {$\bullet$};
      \node at (8.5,2.5) (B0) {$\bullet$};
      \node at (12,3) (B0l) {$\coker \mathbf{B}_0$};
      \draw[-Latex, thick, decorate, decoration={snake, pre length=3pt,post length=3pt}, shorten >= -5pt] (B0l) -- (B0);
      \node at (9.5,-5) (B1l) {$\coker \mathbf{B}_1$};
      \draw[-Latex, thick, decorate, decoration={snake, pre length=3pt,post length=3pt},
      shorten >= -5pt] (B1l) -- (B1);
    \end{tikzpicture}
    \caption{Cartoon of the Quot scheme $\Quot^{t+2}(\mathcal{O}_{\PP^2}^{\!2})$}
    \label{f:21}
  \end{figure}
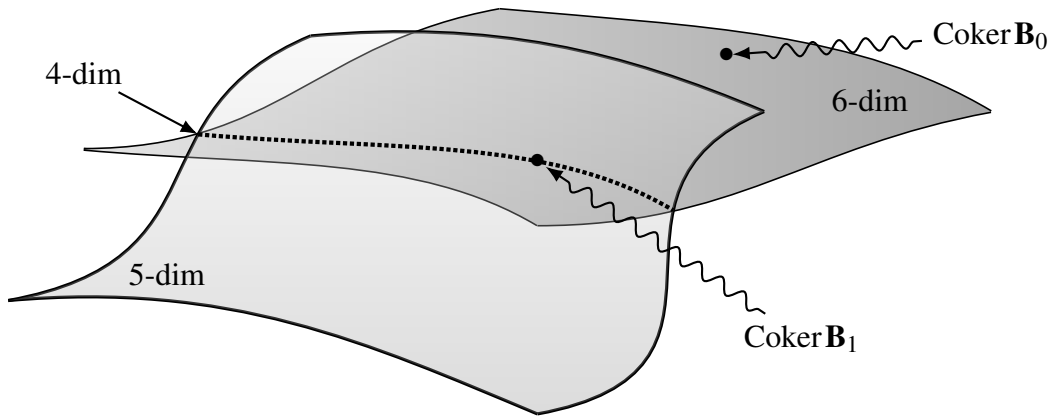
  
  To be more explicit, one irreducible component $X_0$ of the scheme
  $Q^{(2,1)}$ arises as the closure of the locus of points corresponding to
  the cokernels, parametrized over
  $\AA^{\!6} = \Spec\bigl(\kk[z_0,z_1, \dotsc, z_5] \bigr)$, of the
  homogeneous $S$-linear map
  \[
    S^2 
    \xleftarrow{{ \quad
        \mathbf{F}_{0} \coloneqq 
        \begin{bmatrix}
          x_0 + z_1 \, x_2
          & x_1 + z_2 \, x_2 
          & z_0 \, x_2 \\
          z_5 \bigl( -z_3 \, x_1 + (z_1 - z_4 + z_0 z_5) \, x_2 \bigr)
          & z_5(x_1 + z_2 \, x_2) 
          & x_0 + z_3 \, x_1 + z_4 \, x_2   \\
        \end{bmatrix} \quad}}
    S^3(-1)  \, .
  \]
  When $z_0 = z_1 = \dotsb = z_5 = 0$, we recover the matrix $\mathbf{B}_{1}$.
  For any choice of parameters, one verifies that these three columns form a
  Gr\"obner basis (where the monomial order has term over position with the
  lexicographic order on terms in $S$ and position down) for the image of
  $\mathbf{F}_{0}$.  Hence, every element in this family has Hilbert
  polynomial $t+2$ and corresponds to a point in $Q^{(2,1)}$. One also
  confirms that the two relations
  \[
    \mathbf{F}_{0} 
    \begin{bmatrix}
      -z_0 x_2 & 0 \\
      z_0 z_3 x_2 & -z_0 x_2 \\
      x_0 -z_3 x_1 + (z_1 - z_2 z_3) x_2 & x_1 + z_2 x_2 \\
    \end{bmatrix}
    =
    \begin{bmatrix}
      0 & 0 \\
      x_0^2 - z_3^2 x_1^2 + \dotsb 
      & x_0^{} x_1^{} + z_3^{} x_1^2+ \dotsb 
    \end{bmatrix} \, , 
  \]
  together with the columns of $\mathbf{F}_{0}$, form a Gr\"obner basis (where
  the monomial order has position over term with position down and the
  lexicographic order on terms in $S$) for the image of $\mathbf{F}_{0}$.
  Setting $z_1 = z_2 = \dotsb = z_5 = 0$ and letting $z_0 \to \infty$, we
  obtain $\coker \mathbf{B}_0$ as a flat limit of $\coker \mathbf{F}_{0}$.  It
  follows that the lexicographic point lies on $X_0$.
    
  The second irreducible component $X_1$ of the Quot scheme $Q^{(2,1)}$ arises
  as the closure of the locus of points corresponding to the cokernels,
  parametrized over $\AA^{\!5} = \Spec\bigl(\kk[z_2,z_3, \dotsc, z_6] \bigr)$,
  of the homogeneous $S$-linear map
  \[
    S^2 
    \xleftarrow{{ \quad
        \mathbf{F}_{1} \coloneqq 
        \begin{bmatrix}
          x_0 + (z_4 - z_2 z_3) \, x_2 
          & x_1 + z_2 \, x_2 
          & 0 \\
          z_3 (- z_5 \, x_1 - z_6 \, x_2)  
          & z_5 \, x_1 + z_6 \, x_2
          & x_0 + z_3 \, x_1 + z_4 \, x_2   \\
        \end{bmatrix} \quad}}
    S^3(-1)  \, .
  \]
  When $z_2 = z_3 = \dotsb = z_6 = 0$, we recover the matrix $\mathbf{B}_{1}$.
  For any choice of parameters, one verifies that these three columns form a
  Gr\"obner basis (where the monomial order has term over position with the
  lexicographic order on terms and position down) for the image of
  $\mathbf{F}_{1}$.  Hence, every element in this family has Hilbert
  polynomial $t+2$ and corresponds to a point in $Q^{(2,1)}$.

  The intersection $X_0 \cap X_1$ of these two irreducible components is the
  closure of the locus of points corresponding to the cokernels, parametrized
  over $\AA^{\!4} = \Spec\bigl(\kk[z_2, z_3, z_4, z_5] \bigr)$, of the
  homogeneous $S$-linear map
  \[
    S^2 
    \xleftarrow{{ \quad
        \begin{bmatrix}
          x_0 + (z_4 - z_2 z_3) \, x_2 
          & x_1 + z_2 \, x_2 
          & 0 \\
          z_3 z_5 (-x_1 - z_2 \, x_2)  
          & z_5 (x_1 + z_2 \, x_2)
          & x_0 + z_3 \, x_1 + z_4 \, x_2   \\           
        \end{bmatrix} \quad}}
    S^3(-1)  \, .
  \]
  When $z_2 = z_3 = z_4 = z_5 = 0$, we again recover the matrix
  $\mathbf{B}_{1}$. Since the neighbourhood of the unique Borel-fixed point in
  $X_0 \cap X_1$ is an affine space of dimension $4$, this intersection is
  smooth. Moreover, the intersection $X_0 \cap X_1 $ is the divisor defined by
  the ideal $\langle z_6-z_2 z_5 \rangle$ in $X_1$, and is the complete
  intersection defined by the ideal $\langle z_0, z_1-z_4+z_2z_3 \rangle$ in
  $X_0$.
\end{example}

\begin{example}
  \label{e:221}
  Let
  $Q^{(2,2,1)} \coloneqq \smash{\Quot^{2t+2}(\mathcal{O}_{\PP^{2}}^{\!2})}$ be
  the Quot scheme parametrizing quotients of the trivial rank-$2$ vector
  bundle on the projective plane having Hilbert polynomial $2t+2$.  This
  Hilbert polynomial corresponds to the integer partition
  $\lambda \coloneqq (2,2,1)$, because
  $2t+2 = \smash{\binom{t+1}{1} + \binom{t}{1} + \binom{t-2}{0}}$.  The Quot
  scheme $Q^{(2,2,1)}$ has precisely three Borel-fixed points corresponding to
  the cokernels of the homogeneous $S$-linear maps
  \begin{align*}
    S^2
    \xleftarrow{{ \quad
    \mathbf{B}_0 \coloneqq \begin{bmatrix}
      1 & 0 & 0  \\
      0 & x_0^3 & x_0^{2} \, x_1^{}  \\
    \end{bmatrix} \quad}}
        &S \oplus S^2(-3) \, , &
                                 S^2   
        &\xleftarrow{{ \quad
          \mathbf{B}_1 \coloneqq \begin{bmatrix}
            x_{0}^{} & x_{1}^{} & 0  \\
            0 & 0 & x_0^{2}  \\
          \end{bmatrix} \quad}}
    S^1(-1) \oplus S(-2) \, , \\ 
        & \text{and} &
                       S^2
        &\xleftarrow{{ \quad
          \mathbf{B}_2 \coloneqq \begin{bmatrix}
            x_0^{} & 0  \\
            0 & x_0^{} \\
          \end{bmatrix} \quad}}
    S^2(-1)   \, .
  \end{align*}
  The lexicographic point corresponds to $\coker \mathbf{B}_0$.  By
  calculating
  $h^0 \bigl( \Hom_S( \image \mathbf{B}_{i}, \coker \mathbf{B}_{i}) \!\bigr)$
  in \emph{Macaulay2}~\cite{M2} for all $0 \leqslant i \leqslant 2$, we see
  that the tangent spaces at the points determined by $\coker \mathbf{B}_0$,
  $\coker \mathbf{B}_0$, and $\coker \mathbf{B}_1$ have dimension $9$, $10$,
  and $8$ respectively.  It follows that the Quot scheme $Q^{(2,2,1)}$ is
  singular.  A local analysis in a neighbourhood of the Borel-fixed point
  shows that the point corresponding to $\coker \mathbf{B}_1$ lies on two
  distinct irreducible components: the transverse intersection of two affine
  spaces having dimensions $9$ and $8$.  The lexicographic point
  $\coker \mathbf{B}_0$ lies on a unique irreducible component because it is a
  smooth point; see Theorem~\ref{t:smooth}.  A local analysis also shows that
  the point corresponding to $\coker \mathbf{B}_2$ also lies on a unique
  irreducible component having dimension $8$.  We conclude that both
  irreducible components of $Q^{(2,2,1)}$ are smooth.
    
  To be more explicit, one irreducible component $X_0$ of the scheme
  $Q^{(2,2,1)}$ arises as the closure of the locus of points corresponding to
  the cokernels, parametrized over
  $\AA^{\!9} = \Spec\bigl(\kk[z_0,z_1, \dotsc, z_8] \bigr)$, of the transpose
  of the homogeneous $S$-linear map
  \[
    S^2(1) \oplus S(2) 
    \xleftarrow{{\quad
        \mathbf{F}_0^{\textsf{T}} \coloneqq
        \begin{bmatrix}
          x_0^{} + z_4^{} \, x_2^{}
          & z_3^{} (x_0^{} + z_4^{} \, x_2^{}) \\
          x_1^{} + z_2^{} \, x_2^{} 
          & z_3^{} (x_1^{} + z_2^{} \, x_2^{}) \\
          z_0^{} \, x_2^{2} 
          & x_0^2 + z_5^{} \, x_0^{} x_1^{} + z_7^{} \, x_0^{} x_2^{} 
          + z_6^{} \, x_1^2 + z_8^{} \, x_{1}^{} x_{2}^{} + z_1^{} \, x_{2}^{2} \\
        \end{bmatrix} \quad}}
    S^2  \, .
  \]
  When $z_0 = z_1 = \dotsb = z_8 = 0$, we recover the matrix $\mathbf{B}_{1}$.
  For any choice of parameters, one verifies that these three columns form a
  Gr\"obner basis (where the monomial order has term over position with the
  lexicographic order on terms in $S$ and position down) for the image of
  $\mathbf{F}_0$.  Hence, every element in this family has Hilbert polynomial
  $2t+2$ and corresponds to a point in $Q^{(2,2,1)}$. One also confirms that
  the two relations
  \[
    \mathbf{F}_{0} 
    \begin{bmatrix}
      0 & -z_0^{} x_2^{2} \\
      -z_0^{} x_2^{2} & z_0^{} z_5^{} x_2^{2} \\
      x_1 + z_2 x_2 & \!x_0 -z_5 x_1 + (z_4 - z_2 z_5) x_2 \\
    \end{bmatrix}
    \! = \!
    \begin{bmatrix}
      0 & 0 \\
      x_0^{2} x_1^{} + z_2^{} x_0^2 x_2^{} + \dotsb 
      & x_0^3 + (z_4^{}+z_7^{} - z_2^{} z_5^{}) x_0^2 x_2^{} + \dotsb 
    \end{bmatrix} \, , 
  \]
  together with the columns of $\mathbf{F}_{0}$, form a Gr\"obner basis (where
  the monomial order has position over term with position down and the
  lexicographic order on terms in $S$) for the image of $\mathbf{F}_{0}$.
  Setting $z_1 = z_2 = \dotsb = z_8 = 0$ and letting $z_0 \to \infty$, we
  obtain $\coker \mathbf{B}_0$ as a flat limit of $\coker \mathbf{F}_{0}$.  It
  follows that the lexicographic point lies on $X_0$.

  The second irreducible component $X_1$ of the Quot scheme $Q^{(2,2,1)}$
  arises as the closure of the locus of points corresponding to the cokernels,
  parametrized over $\AA^{\!8} = \Spec\bigl(\kk[z_2,z_3, \dotsc, z_9] \bigr)$,
  of the transpose of the homogeneous $S$-linear map
  {\small%
    \[
      S^2(1) \oplus S(2)
      \xleftarrow{{\hspace{-7pt}
          \mathbf{F}_{1}^{\textsf{T}} \coloneqq
          \begin{bmatrix}
            x_0 + z_4 \, x_2 
            & z_3(x_0 + z_4 \, x_2) - z_9 \bigl( z_5 \, x_0 + z_6 \, x_1 
            + (z_8 - z_2 z_6) \, x_2 \bigr) \\
            x_1 + z_2 \, x_2 
            & z_3(x_1 + z_2 \, x_2) 
            + z_9 \bigl( x_0 + (z_7 - z_4 - z_2 z_5) \, x_2 \bigr) \\
            0 
            & \!\!\!\! x_0^2 + z_5^{} \, x_0^{} x_1^{} + z_7^{} \, x_0^{} x_2^{} 
            + z_6^{} \, x_1^2 + z_8^{} \, x_{1}^{} x_{2}^{} + (z_2^{} z_8^{} 
            + z_4^{} z_7^{} - z_4^2 - z_2^2 z_6^{} 
            - z_2^{} z_4^{} z_5^{} ) \, x_{2}^{2} \\
          \end{bmatrix}}}
      S^2  \, .
    \]}%
  When $z_2 = z_3 = \dotsb = z_9 = 0$, we recover the matrix $\mathbf{B}_{1}$.
  For any choice of parameters, one verifies that these three columns form a
  Gr\"obner basis (where the monomial order has position over term with
  position down and the lexicographic order on terms in $S$) for the image of
  $\mathbf{F}_{1}$.  Hence, every element in this family has Hilbert
  polynomial $2t+2$ and corresponds to a point in $Q^{(2,2,1)}$.  One also
  confirms that the columns of $\mathbf{F}_{1}$ form a Gr\"obner basis (where
  the monomial order has term over position with the lexicographic order on
  terms in $S$ and position down) for the image of $\mathbf{F}_{1}$.  Setting
  $z_2 = z_3 = \dotsb = z_8 = 0$ and letting $z_9 \to \infty$, we obtain
  $\coker \mathbf{B}_2$ as a flat limit of $\coker \mathbf{F}_{1}$.  It
  follows that the third Borel-fixed point lies on $X_1$.
    
  The intersection $X_0 \cap X_1$ of these two irreducible components is the
  closure of the locus of points corresponding to the cokernels, parametrized
  over $\AA^{\!7} = \Spec\bigl(\kk[z_2, z_3, \dotsc, z_8] \bigr)$, of the
  transpose of the homogeneous $S$-linear map
  {\small%
    \[
      S^2(1) \oplus S(2) 
      \xleftarrow{{
          \begin{bmatrix}
            x_0 + z_4 \, x_2 
            & z_3(x_0 + z_4 \, x_2) \\
            x_1 + z_2 \, x_2 
            & z_3(x_1 + z_2 \, x_2) \\
            0 
            & \!\!\!\! x_0^2 + z_5^{} \, x_0^{} x_1^{} + z_7^{} \, x_0^{} x_2^{} 
            + z_6^{} \, x_1^2  + z_8^{} \, x_{1}^{} x_{2}^{} 
            + (z_2^{} z_8^{} + z_4^{} z_7^{} - z_4^2 - z_2^2 z_6^{} 
            - z_2^{} z_4^{} z_5^{} ) \, x_{2}^{2} \\
          \end{bmatrix}}}
      S^2  \, .
    \]}%
  When $z_2 = z_3 = \dotsb = z_8 = 0$, we again recover the matrix
  $\mathbf{B}_{1}$. Since a neighbourhood of the unique Borel-fixed point in
  $X_0 \cap X_1$ is an affine space of dimension $7$, this intersection is
  smooth.  Moreover, the intersection $X_0 \cap X_1 $ is the hypersurface
  defined by the ideal $\langle z_9 \rangle$ in $X_1$, and is the complete
  intersection defined by the ideal
  $\langle z_0^{}, z_{1}^{} - z_2^{} z_8^{} - z_{4}^{} z_{7}^{} + z_{4}^{2} +
  z_2^{2} z_{6}^{} + z_{2}^{} z_{4}^{} z_{5}^{} \rangle$ in $X_0$.
\end{example}

Strangely, the two aberrant singularities in the classifications of smooth
Quot schemes and smooth Hilbert schemes do not align perfectly.  The final
remark underscores the differences between the Quot schemes and Hilbert
schemes indexed by the integer partitions $(2,1)$, $(2,2,1)$, and $(2,2,2,1)$.

\begin{remark} 
  \label{r:gap}
  For the integer partition $(2,1)$, the singular Quot scheme
  $Q^{(2,1)} \coloneqq \smash{\Quot^{t+2}(\mathcal{O}_{\PP^{2}}^{\!r})}$,
  where $n \geqslant 2$ and $r \geqslant 2$, is not like the smooth Hilbert
  scheme $\Hilb^{t+2}(\PP^n)$ where $n \geqslant 2$.  This Hilbert scheme has
  a unique Borel-fixed point, and a general point on $\Hilb^{t+2}(\PP^n)$
  corresponds to the disjoint union of a line and a point; see
  \cite{SS23}*{Theorem~3.2}.  However, the irreducible component $X_1$ in
  Example~\ref{e:21} can be regarded as arising from putting the ``line'' and
  the ``point'' in different summands of the bundle
  $\smash{\mathcal{O}_{\PP^{2}}^{\!2}}$.  For example, the Borel-fixed point
  in $Q^{(2,1)}$ corresponding to $\coker \mathbf{B}_1$ has the ideal for a
  line in one summand and the ideal for a point in the other.
    
  For the integer partition $(2,2,1)$, the singular Quot scheme
  $Q^{(2,2,1)} \coloneqq \smash{\Quot^{2t+2}(\mathcal{O}_{\PP^{2}}^{\!2})}$ is
  more like the singular Hilbert scheme $\Hilb^{2t+2}(\PP^n)$, where
  $n \geqslant 3$, than the smooth Hilbert scheme $\Hilb^{2t+2}(\PP^2)$.  When
  $n \geqslant 3$, the Hilbert scheme $\Hilb^{2t+2}(\PP^n)$ has two
  irreducible components: a general point on one corresponds to a pair of skew
  lines and a general point on the other corresponds to the union of a plane
  conic and an isolated point; see \cite{SS23}*{Example~4.4}.  Analogously,
  the irreducible component $X_1$ in Example~\ref{e:221} can be regarded as
  arising from putting two ``lines'' in different summands of the bundle
  $\smash{\mathcal{O}_{\PP^{2}}^{\!2}}$.  For instance, the Borel-fixed point
  in $Q^{(2,2,1)}$ corresponding to $\coker \mathbf{B}_2$ has the ideal of a
  line in each summand.  This reasoning does not apply to
  $\Hilb^{2t+2}(\PP^2)$ because there is not enough elbow room to have two
  disjoint lines in the projective plane.

  For the integer partition $(2,2,2,1)$, the smooth Quot scheme
  $Q^{(2,2,2,1)} \coloneqq \smash{\Quot^{3t+1}(\mathcal{O}_{\PP^{2}}^{\!r})}$,
  where $r \geqslant 2$, is more like the smooth Hilbert scheme
  $\smash{\Hilb^{3t+1}(\PP^2)}$ than the singular Hilbert scheme
  $\smash{\Hilb^{3t+1}(\PP^n)}$, where $n \geqslant 3$.  When $n \geqslant 3$,
  the Hilbert scheme $\smash{\Hilb^{3t+1}(\PP^n)}$ has two irreducible
  components: a general point in one corresponds to a twisted cubic curve and
  a general point in the other corresponds to the union of a plane cubic and
  an isolated point; compare with \cite{PS85}*{Theorem}.  Unlike the disjoint
  union of a line and a point or two skew lines, a twisted cubic curve cannot
  be split over more than one summand of the bundle
  $\smash{\mathcal{O}_{\PP^{2}}^{\!2}}$.
\end{remark}

\subsection*{Acknowledgements}

We thank Joachim Jelisiejew, Paolo Lella, and Ritvik Ramkumar for their
interest and suggestions. Computational experiments done in
\emph{Macaulay2}~\cite{M2} were indispensable.

Gregory G.~Smith was partially supported by NSERC, the NSF grant DMS-1928930
while in residence at the Simons Laufer Mathematical Sciences Institute, and
the Knut and Alice Wallenberg Foundation while visiting the Royal Institute of
Technology (KTH).  Mike Stillman was partially supported by the NSF grant
DMS-2001367, the Simons Foundation (923996, MES), and the Knut and Alice
Wallenberg Foundation while visiting the Royal Institute of Technology (KTH).


\raggedright

\end{document}